\documentclass[12pt]{article}
\usepackage{a4wide, amstext, amsmath, amssymb, graphicx, array, epsfig, psfrag}
\usepackage{amstext, amsmath, amssymb, graphicx}
\usepackage{todonotes}

\def\IR{\mathbb R}
\def\IC{\mathbb C}

\newcommand{\RR}{\mathbb{R}}

\newcommand{\bfn}{\boldsymbol n}
\newcommand{\bfI}{\boldsymbol I}
\newcommand{\bfP}{\boldsymbol P}
\newcommand{\bfx}{\boldsymbol x}

\newcommand{\bfp}{\boldsymbol p}

\newcommand{\bfB}{\boldsymbol B}

\newcommand{\bfkappa}{\boldsymbol \kappa}

\newcommand{\mcV}{\mathcal{V}}
\newcommand{\mcK}{\mathcal{K}}

\newcommand{\mcN}{\mathcal{N}}
\newcommand{\mcF}{\mathcal{F}}

\newcommand{\tn}{|\mspace{-1mu}|\mspace{-1mu}|}

\numberwithin{equation}{section}
\newtheorem{lem}{Lemma}[section]

\newtheorem{thm}{Theorem}[section]
\newtheorem{rem}{Remark}[section]
\newtheorem{cor}{Corollary}[section]
\newcommand{\jump}[1]{[\![#1]\!]}

\newenvironment{proof}{\noindent \newline {\bf Proof.}}
{\hfill \mbox{\fbox{} } \newline}
\newcommand{\nablas}{\nabla_\Sigma}
\newcommand{\nablash}{\nabla_{\Sigma_h}}

\newcommand{\Ps}{{\boldsymbol P}_\Sigma}
\newcommand{\Psh}{{\boldsymbol P}_{\Sigma_h}}
\newcommand{\bfPs}{{\boldsymbol P}_\Sigma}
\newcommand{\bfPsh}{{\boldsymbol P}_{\Sigma_h}}

\newcommand{\ds}{d \sigma}
\newcommand{\dsh}{d \sigma_h}

\begin{document}
\title{\bf A Stable Cut Finite Element Method for Partial Differential
  Equations on Surfaces: The Helmholtz-Beltrami Operator}
\author{Erik Burman\footnote{Department of Mathematics, University College London, London, UK--WC1E 6BT, United Kingdom} \mbox{ }
 Peter Hansbo\footnote{Department of Mechanical Engineering, J\"onk\"oping University,
SE-55111 J\"onk\"oping, Sweden} \mbox{ }
Mats G.\ Larson\footnote{Department of Mathematics and Mathematical Statistics,
  Ume{\AA} University, SE-90187 Ume{\aa}, Sweden} \mbox{ }
Andr\'e Massing\footnote{Department of Mathematics and Mathematical Statistics,
  Ume{\AA} University, SE-90187 Ume{\aa}, Sweden} 
}
\numberwithin{equation}{section} \maketitle
\begin{abstract}
We consider solving the surface Helmholtz equation on a smooth two dimensional 
surface embedded into a three dimensional space meshed with tetrahedra. 
The mesh does not respect the surface and thus the surface cuts through the 
elements.
We consider a Galerkin method based on using the restrictions of continuous 
piecewise linears defined on the tetrahedra to the surface as trial and test functions.

Using a stabilized method combining Galerkin least squares
stabilization and a penalty on the gradient jumps we obtain stability
of the discrete formulation under the condition $h k < C$, where $h$
denotes the mesh size, $k$ the wave number and $C$ a constant
depending mainly on the surface curvature $\kappa$, but not on the surface/mesh intersection.
Optimal error estimates in the $H^1$ and $L^2$-norms follow.
\end{abstract}

\section{Introduction}
In a previous paper~\cite{BHL13}
we considered solving the Laplace-Beltrami problem on a smooth two 
dimensional surface imbedded into a three dimensional space partitioned 
into a mesh consisting of shape regular tetrahedra. The mesh did not 
respect the surface and thus the surface can cut through the elements
in an arbitrary manner.
Following Olshanskii, Reusken, and Grande~\cite{OlReGr09} we constructed a Galerkin method 
by using the restrictions of continuous piecewise linears defined on the tetrahedra 
to the surface.

To alleviate the ill-conditioning of the resulting method we proposed
to add a stabilization term penalizing the jump of the gradient of the
solution to the formulation. The objective of the present work is to
show that in the case of indefinite elliptic problems a similar
stabilization improves the stability of the formulation yielding
discrete wellposedness under a weaker condition on the mesh parameter
and the wave number than is usually expected. The analysis draws on
ideas from \cite{ZhuBuWu16, SwiXX, Wu14} for the stabilization of the
Helmholtz equation. 

The analysis of vibrations and acoustics of thin structures is an
important topic in computational mechanics. Herein we consider, as a
model problem, the surface Helmholtz equation, i.e. the Helmholtz
equation defined using a Laplace-Beltrami operator on the
surface. This problem has many of the difficulties encountered when
using more complex structural models, but is also interesting in its
own right as a model for lateral acoustics in thin
structures. Typically the finite element analysis of the wave equation
in the frequency domain introduces
conditions on the size of the meshsize $h$ compared to the wavenumber~$k$.
For a standard Galeking finite element method of indefinite elliptic
problems, the standard condition that $h k^2$ has to be small, for
stability and
optimal estimates, is
obtained following Schatz \cite{Schatz74}, using the combination of an $H^1$ error estimate by G\aa
rdings inequality and a duality argument showing that the $L^2$-norm
error converges at a faster rate than that measured in the
$H^1$-norm. Thanks to the stabilization the mesh-wavenumber condition takes the form
$hk$ small instead. This condition appears here only
because of the discrete approximation of the surface. Our estimates are explicit in the mesh size and
the wave number, but not in the surface curvature, which we assume
is moderate. The conformity error
introduced due to the approximation of the surface also leads to a
condition on $h$. To simplify the presentation we will assume that 
$k \geqslant 1$ and $h<1$. Generic constants~$C$ may depend on the surface
curvature, but not on the wavenumber, the mesh-size or the
intersection of the surface with the computational mesh.  In cases where we want to highlight a particular
dependence, we add a subscript to the constant.

The outline of the reminder of this paper is as follows: In Section 2
we formulate the model problem and the finite element method, in
Section 3 we prove a priori error estimates, and finally in Section 4
we present numerical investigations confirming our theoretical
results.

\section{Model Problem and Finite Element Method}

\subsection{The Continuous Problem}

Let $\Sigma$ be a smooth two-dimensional closed and orientable surface embedded in ${{\IR}}^3$ with signed 
distance function $b$. We consider the following problem: for a given $k \in \mathbb{R}$, find $u: \Sigma \rightarrow {{\IC}}$ such that
\begin{align}\label{eq:LB}
-\Delta_\Sigma u - k^2 u = f \quad \text{on $\Sigma$}.
\end{align}
Here $\Delta_\Sigma$ is the Laplace-Beltrami operator defined by
\begin{equation}
\Delta_\Sigma = \nabla_\Sigma \cdot \nabla_\Sigma
\end{equation}
where $\nabla_\Sigma$ is the tangent gradient
\begin{equation}
\nabla_\Sigma = \bfPs \nabla
\end{equation}
with $\bfPs = \bfPs(\bfx)$ the projection of $\IR^3$ onto the tangent plane of $\Sigma$
at $\bfx\in\Sigma$, defined by
\begin{equation}
\bfPs = \bfI -\bfn\otimes\bfn
\end{equation}
where $\bfn = \nabla b$ denotes the exterior normal to $\Sigma$ at $\bfx$, $\bfI$ is the identity 
matrix, and $\nabla$ the ${{\IR}}^3$ gradient. 

The corresponding weak statement takes the form: find $u \in H^1(\Sigma)$
such that
\begin{equation}\label{weak_Helmholtz}
a(u,v) = l(v) \quad \forall v \in H^1(\Sigma)
\end{equation}
where
\begin{equation}
a(u,v) = (\nabla_\Sigma u, \nabla_\Sigma v)_\Sigma - (k^2 u,v)_\Sigma, \quad l(v) = (f,v)_\Sigma
\end{equation}
and $(v,w)_\Sigma = \int_\Sigma v \overline w$ is the $L^2$ inner product. We
will assume that $k\in \mathbb{R}$ is such that 
the Fredholm alternative yields a unique
solution of the problem.  Assuming that the following
bound holds on the smallest distance to an eigenvalue of
$\Delta_\Sigma$,
\begin{equation}\label{k2bound}
\min_i |\lambda_i - k^2| \geqslant c k
\end{equation}
we have the following elliptic regularity estimate:
\begin{equation}\label{regularity}
k^{-1} | u |_{2,\Sigma} + |u|_{1,\Sigma} + \|k u\|_{\Sigma} \leqslant C \|f\|_\Sigma.
\end{equation}
Here $\| w \|^2_\Sigma = (w,w)_\Sigma$ denotes the $L^2$ norm on $\Sigma$ and 
\begin{equation}
\| w \|^2_{m,\Sigma} = \sum_{s=0}^m \| ( \otimes_{i = 1}^s \nabla_\Sigma ) w \|_\Sigma^2
\end{equation} 
is the Sobolev norm on $\Sigma$ for $m=0,1,2$, where the $L^2$ norm for a matrix 
is based on the pointwise Frobenius norm. The constant in the above
estimate depends on the curvature of the surface. The following $L^2$-estimate is a
consequence of the Fredholm's 
alternative under the assumption \eqref{k2bound}:
\begin{equation}
\|u\|_{\Sigma} \leqslant  C \max_i |\lambda_i - k^2|^{-1} \|f\|_\Sigma \leq
C k^{-1} \|f\|_\Sigma.
\end{equation}
Using the equation we also immediately obtain a bound of the
$H^1$-norm of $u$
\begin{equation}
 \|\nabla_{\Sigma} u\|_{\Sigma}^2  = (f,u) + k^2 \|u\|_{\Sigma}^2\leqslant C (k^{-1} +
 1) \|f\|^2_{\Sigma}.
\end{equation}
The $H^2$-estimate, finally, is a consequence of the elliptic regularity
of the Laplace-Beltrami operator, $|u|_{2,\Sigma} \leqslant C
\|\Delta_{\Sigma} u\|_{\Sigma}$, see \cite{Aub82}, and the fact that
$
\Delta_\Sigma u = -f - k^2 u
$
implying that
\begin{equation}
\|\Delta_\Sigma u\|_\Sigma \leqslant \|f\|_\Sigma + k^2 \|u\|_\Sigma \leqslant C (1+k) \|f\|_\Sigma.
\end{equation}
\begin{rem}
The assumption \eqref{k2bound} can be checked in special cases such as
for the sphere. In that case $\lambda_i = i (i+1)$, $i=1,2,\hdots$ and
we can see that a moderately small $c$,
for instance $c=0.1$ allows for an important range of values of
$k^2$. The behavior of the method for values of $k^2$ close to an
eigenvalue is explored in section \ref{sec:stab_close_eigenval}.
\end{rem}

\subsection{The Finite Element Method on $\boldsymbol{\Sigma}$}\label{Sec:FEM}

Let $\mcK$ be a quasi uniform partition into shape regular tetrahedra of a domain $\Omega$ in $\IR^3$ completely containing $\Sigma$. Let $\mcK_h$ be the set of tetrahedra that intersect $\Sigma$ and denote by~$\Omega_h$ the domain covered by $\mcK_h$; that is,
\begin{equation}
\mcK_h = \{ K \in \mcK : K\cap \Sigma \neq \emptyset \}, \quad \Omega_h = \cup_{K \in \mcK_h} K.
\end{equation}
We denote the local mesh size by $h_K$ and define the global mesh size $h =\max_{K\in\mcK_h} \{h_K\}$.
Since $h_K \sim h$ by the quasi uniformity of $\mcK$, we will simply use $h$ throughout the remaining work.
We let $\mcV_h$ be the space of continuous
piecewise linear, complex valued, polynomials defined on $\mcK_h$. Our finite element method takes the form:
find  $\tilde u_h \in \mcV_h$ such that
\begin{equation}\label{semi_FEM}
A(\tilde u_h,v) +\gamma_j  j(\tilde u_h,v)= l_s(v) \quad \forall v \in \mcV_h
\end{equation}
where the bilinear form $A(\cdot,\cdot)$ is defined by
\begin{equation}
A(v,w) = a(v,w) + \gamma_s s(v,w)\quad \forall v,w \in \mcV_h
\end{equation}
with the stabilization terms
\begin{equation}
s(v,w) = \sum_{K \in \mcK_h} h^2 (\Delta_\Sigma v + k^2 v,
\Delta_\Sigma w + k^2 w)_{\Sigma \cap K}
\end{equation}
and
\begin{equation}
j(v,w) = \sum_{F \in \mcF_I} ([\bfn_F \cdot \nabla v],[\bfn_F \cdot \nabla w])_F.
\end{equation}
Above $\mcF_I$ denotes the set of internal interfaces in $\mcK_h$,
$[\bfn_F \cdot \nabla v] = (\bfn_F \cdot \nabla v)^+ 
- (\bfn_F \cdot \nabla v)^-$ with 
$w(\bfx)^\pm = \lim_{t\rightarrow 0^+} w(\bfx \pm t \bfn_F)$, is 
the jump in the normal gradient across the face $F$, and $\bfn_F$ 
denotes a fixed unit normal to the face $F$. For consistency the right
hand side is modified to read
\begin{equation}
 l_s(v) = l(v) - \sum_{K \in \mcK_h} \gamma_s h^2 (f,
\Delta_\Sigma v + k^2 v)_{\Sigma \cap K}.
\end{equation}
The parameter $\gamma_x \in \mathbb{C}$, $x=s,j$ will be assumed to satisfy
$\mbox{Im}(\gamma_x) > 0$. To simplify the presentation and without loss
of generality we will also assume that $Re(\gamma_x)=0$ below.

\subsection{Approximation of the Surface}

Next, we recall that for a smooth oriented surface $\Sigma$,
there is an open $\delta$ tubular neighborhood
$U_{\delta}(\Sigma) =\{ \bfx \in \RR^3 : |b(\bfx)| < \delta \} $ of
$\Sigma$ such that for each $\bfx \in U_{\delta}(\Sigma)$ there is a unique
closest point $\bfp(\bfx) \in \Sigma$
minimizing the Euclidean distance to $\bfx$. Note that
the closest point mapping $\bfx \mapsto \bfp(\bfx)$ satisfies
$\bfp(\bfx) = \bfx - b(\bfx) \bfn(\bfp(\bfx))$.
Using $\bfp$ we extend $u$ outside of
$\Sigma$ by defining
\begin{equation}
  \label{extension} u^e(\bfx) = u.
  \circ \bfp (\bfx)
\end{equation}
In the following, a superscript $e$ is also used to denote the
extension of other quantities defined on the surface.

In practice we are typically not able to compute on the exact surface $\Sigma$, instead 
we have to consider an approximate surface $\Sigma_h$. Depending on how the surface 
is described the construction of the approximate surface can be done in different ways. 
Here we consider, in particular, a simple situation where $\Sigma$ is described 
by a level set function $b$ and $\Sigma_h$ is defined by the zero level set to a 
piecewise linear approximate level set function $b_h \in \mbox{Re}(\mcV_h)$.
In this case the approximate 
surface is a piecewise linear surface since it is the level set to a piecewise 
linear function.
We let the approximate normal $\bfn_h$ be the exact normal to the 
piecewise linear approximate surface $\Sigma_h$.
and that the following estimates hold
\begin{equation}
\| b \|_{L^\infty(\Sigma_h)} \leqslant C h^2, 
\qquad \| \bfn^e - \bfn_h \|_{L^\infty(\Sigma_h)} \leqslant C h.
\end{equation}
These properties are, for instance, satisfied if $b_h$ is the Lagrange interpolant of $b$.
Observe that by the properties of the interpolant the discrete interface $\Sigma_h$ is also contained in $\mcK_h$.
Finally, we define the lift $v^l$ of a function $v$ defined on discrete surface $\Sigma_h$
to the exact surface $\Sigma$ by requiring that
\begin{equation}
  \label{eq:lift}
  (v^l)^e = v^l \circ \bfp = v.
\end{equation}
We refer to Figure~\ref{fig:domain-set-up} for an illustration of the relevant geometric concepts. 
\begin{figure}[htb]
  \begin{center}
    \includegraphics[width=0.45\textwidth]{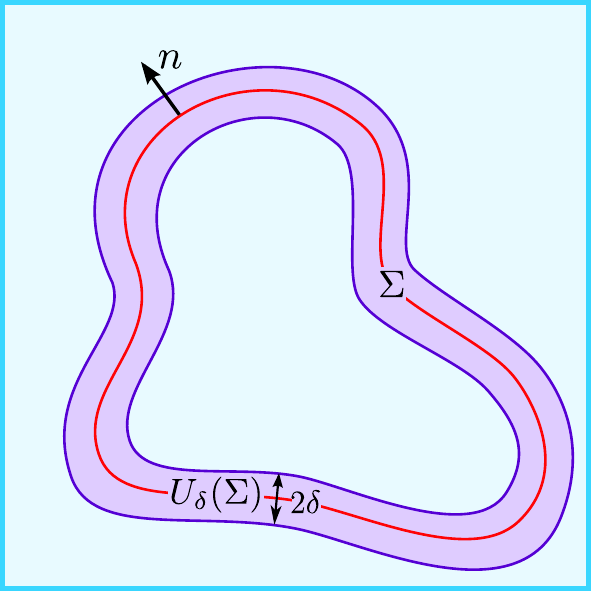}
    \hspace{0.03\textwidth}
    \includegraphics[width=0.45\textwidth]{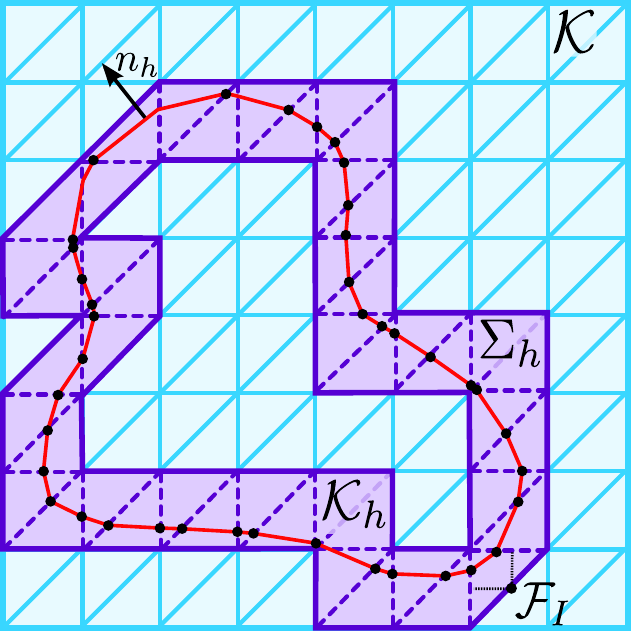}
  \end{center}
  \caption{Set-up of the continuous and discrete domains. (Left) Continuous surface $\Sigma$
  enclosed by a $\delta$ tubular neighborhood $U_{\delta}(\Sigma)$.
  (Right) Discrete manifold $\Sigma_h$ embedded into a background mesh
  $\mcK$ from which the active mesh $\mcK_h$ is extracted.
  }
  \label{fig:domain-set-up}
\end{figure}

\subsection{The Finite Element Method on $\boldsymbol{\Sigma}_{\boldsymbol{h}}$}

Here let
\begin{equation}
\mcK_h = \{ K \in \mcK : K \cap \Sigma_h \neq \emptyset \}, \quad \Omega_h = \cup_{K \in \mcK_h} K
\end{equation}
and $\mcV_h$ be the continuous piecewise linear, complex valued functions defined on $\mcK_h$. The 
finite element method on $\Sigma_h$ takes the form: find $u_h \in \mcV_h$ such that
\begin{equation}\label{full_FEM}
A_h(u_h,v)  + \gamma_j j(u_h,v) = l_h(v) \quad \forall v \in \mcV_h.
\end{equation}
The bilinear form $A_h(\cdot, \cdot)$ is defined by
\begin{equation}
A_h(v,w) = a_h(v,w) +  \gamma_s s_h(v,w)\quad \forall v,w \in \mcV_h
\end{equation}
with
\begin{equation}
a_h(v,w) = (\nabla_{\Sigma_h} v, \nabla_{\Sigma_h} w)_{\Sigma_h} -
(k^2 v, w)_{\Sigma_h}
\end{equation}
and 
\begin{equation}
s_h(v,w) = \sum_{K \in \mcK_h}  h^2 (\Delta_{\Sigma_h} v + k^2
v,\Delta_{\Sigma_h} w + k^2 w)_{\Sigma_h \cap K}
\end{equation}
where the tangent gradients are defined using the normal to the discrete 
surface
\begin{equation}
\nabla_{\Sigma_h} v 
= \bfPsh \nabla v = ( \bfI - \bfn_h \otimes \bfn_h ) \nabla v.
\end{equation}
The form on the right hand side $l_h(\cdot)$ is given by
\begin{equation}
l_h(v) 
= ( f^e, v )_{\Sigma_h}- \sum_{K \in \mcK_h} \gamma_s h^2 (f^e,\Delta_{\Sigma_h} w  +k^2 w)_{\Sigma_h \cap K}.
\end{equation}
Observe that since the level set function $b_h$ is piecewise linear and
defined on $\mcV_h$, $\Delta_{\Sigma_h} v\vert_{K \cap \Sigma_h} =
0$. Therefore the
stabilization term and the right hand side reduces to
\begin{equation}
s_h(v,w) =  \sum_{K \in \mathcal{K}_h} h^2 (\Delta_{\Sigma_h} v + k^2 v,\Delta_{\Sigma_h} w +
k^2 w)_{\Sigma_h} = (h^2 k^2 v,
k^2 w)_{\Sigma_h}
\end{equation}
and 
\begin{equation}
l_h(v) = ( f^e, v )_{\Sigma_h}- \sum_{K \in \mcK_h} \gamma_s h^2 (f^e,\Delta_{\Sigma_h} v +
k^2 v)_{\Sigma_h \cap K} = ( f^e, v - \gamma_s h^2 k^2 v)_{\Sigma_h} 
\end{equation}
We notice that these simplifications allow us to write the following
formulation which is suitable for implementation: find $u_h \in
\mcV_h$ such that
\begin{equation}
(\nabla_{\Sigma_h} u_h, \nabla_{\Sigma_h} v)_{\Sigma_h} -
(k^2(1-\gamma_s h^2) u_h, v)_{\Sigma_h} + \gamma_j j(u_h,v) = (f_e,
(1-\gamma_s h^2k^2) v)_{\Sigma_h}\quad \forall v \in \mcV_h.
\end{equation}
Since this weakly consistent stabilization actually is a norm on $u_h$,
one may prove that the system is invertible for
all $h$ as follows. Take $v = u_h$ in \eqref{full_FEM} and
take the imaginary part of the equation to obtain
\begin{equation}
\mbox{Im}(\gamma_s) h^2 \| k u_h\|^2_{\Sigma_h} \leqslant \|f\|_{\Sigma_h} (1
+ |\gamma_s| (hk)^2) \|u_h\|_{\Sigma_h}.
\end{equation}
Therefore
\begin{equation}
\|u_h\|_{\Sigma_h} \leqslant \mbox{Im}(\gamma_s)^{-1} ((hk)^{-2} + 1) \|f\|_{\Sigma_h}.
\end{equation}
As we shall see below, the lack of consistency introduces some
additional constraints on the stabilization parameters.

The penalty on the gradient jumps is necessary to obtain
robustness in the semi-discrete case, but also to control the
conformity error of the stabilizing terms in the fully discrete case. We recall the following
key result from \cite{BHL13}.
\begin{lem}\label{grad_cont}
There exists $C>0$ so that for all $v_h \in \mcV_h$ there holds
\begin{equation}
h \|\nabla v_h\|_{\Sigma_h} \leqslant C( h \|\nabla_{\Sigma_h} v_h\|_{\Sigma_h}
+ j(v_h,v_h)^{\frac12}).
\end{equation}
\end{lem}
\begin{proof}
Identical to the proof of Lemma 3.2 of \cite{BHL13}.
\end{proof}
\section{A Priori Error Estimates}
For the a priori error analysis we will follow the framework for the
analysis of stabilized finite element methods for the Helmholtz
equation proposed in \cite{SwiXX}. In order to estimate the error
induced by approximating the equations on an approximate surface we
need to first recall a number of technical results regarding the
mapping from the approximate to the exact surface and the bounds on
the error committed when changing the domain of integration. For
detailed proofs, we refer
to~\cite{Dziuk1988,OlReGr09,BHL13}. We
also recall some approximation error estimates.

\subsection{Geometric Estimates}
First we recall how the tangential gradient of lifted and extended functions can be computed and
how the surface measure changes under lifting.
Starting with the Hessian of the signed distance function
\begin{align}
  \bfkappa = \nabla \otimes \nabla b \quad \text{on }
  U_{\delta_0}(\Sigma)
\end{align}
the derivative of the closest point projection 
and of an extended function $v^e$ is given by
\begin{gather}
Dp = \Ps (I - b \bfkappa) = \Ps - b \bfkappa
\label{eq:derivative-closest-point-projection}
\\
  Dv^e = D(v \circ \bfp) = Dv Dp = Dv \bfP_{\Sigma}(I - b \bfkappa).
\label{eq:derivative-extended-function}
\end{gather}
The self-adjointness of $\Ps$, $\Psh$, and $\bfkappa$,
and the fact that $ \Ps \bfkappa = \bfkappa = \bfkappa \Ps$
and $\Ps^2 = \Ps$
leads to the identity
\begin{align}
  \nablash v^e &= \Psh(I - b \bfkappa)\Ps \nabla v = \bfB^{T} \nablas v
  \label{eq:ve-tangential-gradient}
\end{align}
where $\bfB$ denotes the invertible linear application
\begin{align}
  \bfB = \bfP_{\Sigma}(I - b \bfkappa) \bfP_{\Sigma_h}: T_x(\Sigma_h) \to T_{\bfp(x)}(\Sigma)
  \label{eq:B-def}
\end{align}
mapping the tangential space of $\Sigma_h$ at $x$ to the tangential space of $\Sigma$ at
$\bfp(x)$. Setting $v = w^l$ and using the identity $(w^l)^e = w$, we immediately get that
\begin{align}
  \nablas w^l = \bfB^{-T} \nablash w
\end{align}
for any elementwise differentiable function $w$ on $\Sigma_h$ lifted to $\Sigma$.
We recall from \cite[Lemma 14.7]{GilbargTrudinger2001}
that for $x\in U_{\delta_0}(\Sigma)$, the Hessian $\bfkappa$
admits a representation
\begin{equation}\label{Hform}
  \bfkappa(x) = \sum_{i=1}^d \frac{\kappa_i^e}{1 + b(x)\kappa_i^e}a_i^e \otimes a_i^e
\end{equation}
where $\kappa_i$ are the principal curvatures with corresponding
principal curvature vectors $a_i$.
Thus
\begin{equation}
  \|\bfkappa\|_{L^\infty(U_{\delta_0}(\Sigma))} \leqslant C 
  \label{eq:Hesse-bound}
\end{equation}
for $\delta_0 > 0$ small enough and as a consequence
the following bounds for the linear
operator $\bfB$ can be derived:
  \begin{equation}
    \| \bfB \|_{L^\infty(\Sigma_h)} \leqslant C ,
    \quad \| \bfB^{-1} \|_{L^\infty(\Sigma)} \leqslant C,
    \quad
    \| \bfP_\Sigma - \bfB \bfB^T \|_{L^\infty(\Sigma)} \leqslant C  h^2.
    \label{eq:BBTbound}
  \end{equation}

Next, we recall that the surface measure $\ds$ on $\Sigma$ is related to the surface measure $\dsh$ 
on $\Sigma_h$ by the identity  
\begin{equation}
d \sigma = |\bfB | d \sigma_h
\end{equation}
where $|\bfB|$ is the determinant of $\bfB$ which is given by
\begin{equation}
|\bfB| = \Pi_{i=1}^2 (1 - b \kappa_i^e) \bfn^e \cdot \bfn_h.
\end{equation}
Using this the following estimates for the determinant can be proved,
\begin{equation}\label{detBbounds}
\|\, |\bfB|\,\|_{L^\infty(\Sigma_h)} \leqslant C,
\quad
\|\, |\bfB|^{-1} \,\|_{L^\infty(\Sigma_h)} \leqslant C,
\quad
 \|1 - |\bfB| \|_{L^\infty(\Sigma_h)} \leqslant C h^2.
\end{equation}

\subsection{Interpolation Error Estimates}

In order to define an interpolation operator we note
that thanks to the coarea-formula 
\begin{align*}
\int_{U_{\delta}} f(x) \,dx = \int_{-\delta}^{\delta} 
\left(\int_{\Sigma(r)} f(y,r) \, \mathrm{d} \Sigma_r(y)\right)\,\mathrm{d}r
\end{align*}
see, e.g., \cite{EvansGariepy1992}, the extension $v^e$ of $v \in H^s(\Sigma)$ satisfies the stability estimate
\begin{equation}\label{eq:ext_stab}
\| v^e \|_{s,\Omega_h} \leqslant C h^{\frac12}  \| v \|_{s,\Sigma},\quad s=0,1,2.
\end{equation}
For $h$ sufficiently small the constant in the inequality \eqref{eq:ext_stab} 
depends only on the curvature of the surface $\Sigma$. The above dependence on $h$ 
can be obtained by mapping $\Omega_h$ to some reference 
shell where both the diameter and the thickness are fixed. On this domain the standard 
result for extension operators $\| E v \|_{s,\Omega_h} \leqslant C \| v \|_{s,\Sigma}$ holds 
and \eqref{eq:ext_stab} follows by scaling back to the physical domain noting that the thickness, in the direction normal to $\Sigma$, of $\Omega_h$ is $O(h)$. 

We let $\pi_h: L^2(\Omega_h) \rightarrow \mcV_h|_{\Sigma_h}$ denote the standard 
Scott-Zhang interpolation operator and recall the interpolation error estimate
\begin{equation}\label{interpolstandard}
\| v - \pi_h v \|_{m,K} \leqslant C h^{2-m} \| v \|_{2,\mcN(K)}, \quad m = 0,1,2
\end{equation}
where $\mcN(K)\subset \Omega_h$ is the union of the neighboring elements of $K$. We 
also define an interpolation operator 
$\pi_h^l:L^2(\Sigma) \rightarrow (\mcV_h|_{\Sigma_h})^l$ 
as follows
\begin{equation}\label{pihl}
\pi_h^l v =  ( (\pi_{h} v^e) |_{\Sigma_h})^l.
\end{equation}

Introducing the energy norm $\tn \cdot \tn_\Sigma$ associated with the 
exact surface and the energy norm $\tn \cdot \tn_{\mcF}$  associated with the 
jump terms
\begin{equation}
\tn v \tn^2_{\Sigma,k} = \|\nabla_{\Sigma} v\|^2_{\Sigma}+\| k v\|^2_{\Sigma}, \quad \tn v
\tn_{\mcF}^2 = j(v,v), \quad \tn v \tn_{GLS}^2 = s(v,v).
\end{equation}
From the results of \cite{BHL13} we deduce approximation results
needed in the analysis.
\begin{lem}\label{lem:approx} The following estimates hold
\begin{equation}\label{interpol}
\tn u - \pi_h^l u \tn_{\Sigma,k}^2 + \tn u^e - \pi_h u^e \tn^2_{\mcF} + \sum_K \|h^{-\frac12} (u - \pi_h^l u)\|^2_{\partial K
  \cap \Sigma} \leqslant C h^2 (1 + h^4 k^4) | u |^2_{2,\Sigma}
\end{equation}
\begin{equation}\label{interp_GLS}
\tn u - \pi_h^l u \tn^2_{GLS} \leqslant C h^2 (1 + h^4 k^4)   \| u \|^2_{2,\Sigma}
\end{equation}
and, 
\begin{equation}\label{pihlbound}
\tn \pi_h u^e \tn_{\mcF}^2 + h^2 \tn \pi_h^l u \tn_{\Sigma,k}^2 + \tn
\pi_h^l u \tn_{GLS}^2\leqslant C h^2  ( (1 + h^4 k^4)\| u\|^2_{2,\Sigma}+ \|f\|^2_{\Sigma}).
\end{equation}
\end{lem}
\begin{proof}
The bound \eqref{interpol}
 follows immediately from
the approximation results of \cite{BHL13}. For \eqref{interp_GLS}
 we use the following relation
that follows from the arguments in \cite{LaLa13}: since $u_h$ is piecewise affine there
holds
\begin{equation}\label{eq:Delta_bound}
\Delta_\Sigma u_h^l\vert_{\Sigma \cap K}  = - \mbox{tr}( \kappa)
\nabla u_h^l \cdot n_{\Sigma}\vert_{\Sigma \cap K}, \forall K \in \mcK_h.
\end{equation}
To see this we write
\begin{multline*}
\Delta_\Sigma u_h^l = \nabla_{\Sigma} \cdot (I - n_{\Sigma} \otimes
n_{\Sigma}) \nabla u_h^l = \nabla_{\Sigma} \cdot  (n_{\Sigma}
n_{\Sigma}\cdot \nabla
u_h^l )  = -(\nabla \cdot n_\Sigma) n_{\Sigma}\cdot \nabla
u_h^l - \underbrace{ n_{\Sigma} \cdot (\nabla_\Sigma n_{\Sigma}) (\nabla u_h^l)}_{=0}
\end{multline*}
and the relation follows recalling that $\nabla \cdot n_\Sigma = \mbox{tr}( \kappa)$.
We may then use the triangle
inequality to obtain
\begin{multline*}
\sum_{K \in \mcK}\|\Delta (u - \pi_h^l u)+k^2  (u - \pi_h^l
u)\|_{\Sigma \cap K}^2 \leq C(\|\Delta_\Sigma u\|^2_\Sigma + \|\mbox{tr}( \kappa)
\nabla (\pi_h^l u_h - u^e) \cdot n_{\Sigma}\|_\Sigma^2+\tn u - \pi_h^l u \tn_{\Sigma,k}^2 )\\ \leq  C (1+ k^4 h^4) \|u\|_{2,\Sigma}^2.
\end{multline*}
To prove \eqref{pihlbound} we add and subtract $u$, use a triangle
inequality and apply \eqref{interpol} and \eqref{interp_GLS} and
finally observe that, using the regularity \eqref{regularity} and the
equation \eqref{eq:LB},
\[
 h^2 \tn u \tn_{\Sigma,k}^2+\tn u \tn_{GLS}^2 \leq C h^2  \|f\|_{\Sigma}^2.
\]
\end{proof}

\subsection{Error Estimates for the Semi Discretized Formulation}
We will first give an analysis for the semi-discretized method
\eqref{semi_FEM}. This is to show how the ideas of \cite{SwiXX} carries
over to the case of approximation of the Helmholtz equation on a
surface, without the technicalities introduced by the discretized surface.
The analysis is based on the observation that we have coercivity on
the stabilization terms that constitute a (very weak) norm on the
solution. In this norm we obtain an optimal error estimate. We then
proceed using duality to estimate the error in the $L^2$-norm,
independent of the error in energy norm. Then finally we estimate the
error in the energy norm. Since the two stabilization terms have
similar effect in this case we use the generic parameter $\gamma =
\gamma_s = \gamma_j$. To simplify the notation we assume that $hk$ is
bounded by some constant, so that higher powers can be
omitted. Observe however that we do not assume that $hk$ is ``small
enough'' here, which will be necessary when also the domain is
discretized in the next section.
We first prove a preliminary lemma that will
be useful in the following analysis.

\begin{lem}\label{cont_lem}(Continuity)
For all $v,w \in H^2(\Sigma)$, $v_h,w_h \in \mcV_h$ there holds
\begin{multline}\label{cont1}
|a(v+v_h,w+w_h)| \leqslant \tn v + v_h \tn_{GLS} \|h^{-1}(w+w_h)\|_{\Sigma}
\\
+ C \tn v^e + v_h \tn_{\mcF} \left(\sum_K \|h^{-\frac12} (w + w_h)\|^2_{\partial K
  \cap \Sigma}\right)^{\frac12}.
\end{multline}
\end{lem}
\begin{proof}
Using an integration by parts we see that
\begin{multline}
a(v+v_h,w+w_h)| = \sum_K \int_{\partial K
  \cap \Sigma} \jump{\nabla_{\Sigma} v_h} \cdot n_{\partial K \cap
  \Sigma} \overline{ (w+w_h) }~\mbox{d}\sigma \\
- \sum_K (\Delta_{\Sigma} (v + v_h) +k^2
(v + v_h) , w+w_h)_{K \cap \Sigma}.
\end{multline}
We now multiply and divide and with $h^{\frac12}$ in the first term of the
right hand side and with $h$ in the second. Then we apply the Cauchy-Schwarz
inequality and observe that by using
trace inequalities from $\partial K \cap \Sigma$ to $F \in \partial K$,
\begin{equation}
\sum_K (h \jump{\nabla_{\Sigma} v_h}, \jump{\nabla_{\Sigma} v_h}) _{\partial K
  \cap \Sigma} \leqslant C \tn v_h \tn^2_{\mcF} = C \tn v^e+v_h \tn^2_{\mcF}.
\end{equation}
This completes the proof of \eqref{cont1}.
\end{proof}
\begin{rem}
Observe that by the symmetry of the form $a(\cdot,\cdot)$ the claim
holds also when $v,v_h$ and $w,w_h$ are interchanged.
\end{rem}
\begin{lem}\label{stab_conv}
Let $u$ be the solution of \eqref{weak_Helmholtz} and $\tilde u_h$ the solution of
\eqref{semi_FEM}. Assume that the regularity estimate \eqref{regularity} holds then
\begin{equation}
\tn u - \tilde u_h \tn_{GLS} + \tn u^e - \tilde u_h \tn_{\mcF} \leqslant C \mbox{Im}(\gamma)^{-1}  (|\gamma| +
1) hk \|f\|_{\Sigma}.
\end{equation}
\end{lem}
\begin{proof}
By the condition $\mbox{Im}(\gamma)>0$, and the regularity of $u$ we note that there holds
\begin{align}\nonumber
&\mbox{Im}(\gamma) (\tn u - \tilde u_h \tn^2_{GLS} + \tn u^e - \tilde u_h
\tn^2_{\mcF} ) 
\\
&\qquad = \mbox{Im}(A(u - \tilde u_h,u - \tilde u_h)+\gamma j(u^e - \tilde u_h,u^e - \tilde u_h)).
\end{align}
Using now the consistency of the formulation we have by Galerkin
orthogonality
\begin{align} \nonumber
&\mbox{Im}(\gamma) (\tn u - \tilde u_h \tn^2_{GLS} + \tn u^e - \tilde u_h
\tn^2_{\mcF} ) 
\\
&\qquad = \mbox{Im}(A(u - \tilde u_h,u - \pi_h u^e ) +\gamma
j(u^e - \tilde u_h,u^e -\pi_h u^e))
\\
\label{error_rep}
&\qquad \leqslant  |a(u - \tilde u_h,u - \pi_h u^e ) + \gamma s(u - \tilde u_h,u -
\pi_h u^e ) + \gamma j(u^e - \tilde u_h,u^e - \pi_h u^e )|.
\end{align}
By Lemma \ref{cont_lem} there holds
\begin{multline}\label{continuity}
|a(u - \tilde u_h,u - \pi_h u^e )|\leqslant \tn u - \tilde u_h
\tn_{GLS} \|h^{-1} (u - \pi_h u^e)\|_{\Sigma}
\\
+ C \tn u^e - \tilde u_h
\tn_{\mcF} \left(\sum_K \|h^{-\frac12} (u - \pi_h u^e)\|^2_{\partial K
  \cap \Sigma}\right)^{\frac12}.
\end{multline}
For the stabilization terms we use the Cauchy-Schwarz inequality to
obtain
\begin{multline}
|\gamma s(u - \tilde u_h,u -
\pi_h u^e ) + \gamma j(u^e - \tilde u_h,u^e - \pi_h u^e )| \\
\leqslant |\gamma|  (\tn u - \tilde u_h \tn_{GLS} + \tn u^e - \tilde u_h
\tn_{\mcF} )  (\tn u - \pi_h u^e\tn_{GLS} + \tn u^e - \pi_h u^e
\tn_{\mcF} ).
\end{multline}
The claim now follows by applying Lemma \ref{lem:approx} and the
regularity estimate \eqref{regularity}.
\end{proof}
\begin{thm}\label{semi_conv}
Let $u$ be the solution of \eqref{weak_Helmholtz} and $\tilde u_h$ the solution of
\eqref{semi_FEM}. Assume that the regularity estimate \eqref{regularity} holds. Then
\begin{equation}
\tn u - \tilde u_h \tn_{\Sigma,k} \leqslant C_{\gamma} \mbox{Im}(\gamma)^{-1}  (|\gamma| +
1) (hk + h^2 k^3)  \|f\|_{\Sigma}.
\end{equation}
\end{thm}
\begin{proof}
First let $z$ be the solution of \eqref{weak_Helmholtz} with the right hand
side $f = u-\tilde u_h$. Then by the finite element formulation \eqref{semi_FEM} there holds
\begin{equation}
\|u - \tilde u_h \|^2_\Sigma 
= a(u - \tilde u_h, z - \pi_h z^e) - \gamma s(u -
\tilde u_h,\pi_h z^e) - \gamma j(u^e -
\tilde u_h,\pi z^e).
\end{equation}
Using Lemma \ref{cont_lem} we obtain the bound
\begin{align}
\|u - \tilde u_h \|^2_\Sigma 
&\leqslant \tn u - \tilde u_h \tn_{GLS}
\|h^{-1}(z - \pi_h z^e)\|_{\Sigma} 
\\
&\qquad + C\tn u^e - \tilde u_h
\tn_\mcF \left(\sum_K \|h^{-\frac12} (z - \pi_h z^e)\|^2_{\partial K
  \cap \Sigma}\right)^{\frac12}
\\
&\qquad + |\gamma|(\tn u - \tilde u_h \tn_{GLS} + \tn u^e - \tilde u_h \tn_\mcF)(\tn \pi_h z^e \tn_{GLS} + \tn \pi_h z^e \tn_\mcF).
\end{align}
By interpolation, the definition of $z$ and the regularity of $z$ we
obtain
\begin{equation}
\|h^{-1}(z - \pi_h z^e)\|_{\Sigma} \leqslant C hk \|u -
\tilde u_h\|_{\Sigma}
\end{equation}
\begin{equation}\label{stab_dual1}
\tn \pi_h z^e \tn_{GLS} \leqslant \tn \pi_h z^e -
z\tn_{GLS} + C h \|u -
\tilde u_h\|_{\Sigma} \leqslant C h (1 + k )\|u -
\tilde u_h\|_{\Sigma}
\end{equation}
and
\begin{equation}\label{stab_dual2}
\tn \pi_h z^e \tn_\mcF = \tn \pi_h z^e - z^e\tn_\mcF \leq
C hk \|u - \tilde u_h\|_{\Sigma}.
\end{equation}
Collecting the above bounds and using Lemma \ref{stab_conv} we obtain
\begin{equation}\label{L2error}
\|k (u - \tilde u_h) \|_\Sigma \leqslant C ( 1 + |\gamma|)  hk^2 (\tn u - \tilde u_h \tn_{GLS} +
\tn u^e - \tilde u_h \tn_\mcF) \leqslant C_{\gamma} h^2k^3
\|f\|_{\Sigma}.
\end{equation}
We may now proceed to bound $\tn u - \tilde u_h \tn^2_{\Sigma,k}$  using
the real part of the bilinear form, Galerkin orthogonality, and the
control of the $L^2$-norm of the error.
\begin{equation}\label{H1semi}
\tn u - \tilde u_h \tn^2_{\Sigma,k} = \mbox{Re}(A(u - \tilde u_h,u -  \pi_h u^e) -
\gamma j(\tilde u_h,\tilde u_h -  \pi_h u^e)) + 2 \|k (u - \tilde u_h) \|_\Sigma^2.
\end{equation}
In the first term of the right hand side we now proceed as for
\eqref{error_rep} using the inequality
\eqref{continuity} and Lemma \ref{stab_conv} to conclude that
\begin{equation}
|A(u - \tilde u_h,u -  \pi_h u^e) -
\gamma j(\tilde u_h,\tilde u_h -  \pi_h u^e)| \leqslant C_{\gamma} (hk)^2 \|f\|_\Sigma^2.
\end{equation}
We conclude by combining this bound with \eqref{L2error}.
\end{proof}
\begin{lem}\label{apriori_bound}
Under the same assumptions as for Lemma \ref{stab_conv} and Theorem
\ref{semi_conv} there holds
\begin{equation}\label{L2error2}
\|u - \tilde u_h\|_\Sigma \leqslant C_{\gamma} (h k)^2 \|f\|_{\Sigma}
\end{equation}
and
\begin{equation}
\tn \tilde u_h \tn_{\Sigma,k} \leqslant C_{\gamma}
 (1 + k) \|f\|_{\Sigma},\quad \tn \tilde u_h \tn_{GLS} \leqslant C_{\gamma}
 (1 + k) h \|f\|_{\Sigma}.
\end{equation}
\end{lem}
\begin{proof}
The first claim follows directly from equation \eqref{L2error}.
The remaining inequalities are immediate by adding and subtracting the exact solution $u$ in the
norms of the left hand side, followed by a triangle inequality and then applying the results of Lemma
\ref{stab_conv} and Theorem \ref{semi_conv}.
\end{proof}
\subsection{Error Estimates for the Fully Discrete Formulation}
To obtain an error estimate for the fully discrete scheme we need a
equivalent to Lemma \ref{cont_lem} for the formulation on the discrete
surface and we also need
upper bounds of the conformity error that we commit by approximating the
surface. We start by proving these technical lemmas.
\begin{lem}\label{cont_lem2}(Continuity)
For all $v,w \in H^2(\Sigma)$, $v_h,w_h \in \mcV_h$ there holds
\begin{multline}
  \label{cont2}
|a(v+v^l_h,w+w^l_h)| \leqslant \tn v + v^l_h \tn_{GLS} \|h^{-1}(w+w^l_h)\|_{\Sigma}
\\
+ C (\tn v^e + v_h \tn_{\mcF}  + h \tn v_h^l \tn_{\Sigma,k})\left(\sum_K \|h^{-\frac12} (w + w^l_h)\|^2_{\partial K
    \cap \Sigma}\right)^{\frac12}.
\end{multline}
\end{lem}
\begin{proof}
The proof of \eqref{cont2} is similar to that of \eqref{cont1}, but this time we instead need
to prove the inequality
\begin{equation}
\sum_K (h \jump{\nabla_{\Sigma} v^l_h}, \jump{\nabla_{\Sigma} v^l_h}) _{\partial K
  \cap \Sigma} \leqslant C \tn v_h \tn^2_{\mcF} = C \tn v^e+v_h \tn^2_{\mcF}
\end{equation}
to conclude. This leads to a slightly different argument since $\nabla_\Sigma v_h^l = \bfB^{-T} \bfPsh \nabla v_h$. It follows
that
\begin{equation}
\sum_{K \in \mathcal{T}_h} \int_{\Sigma \cap \partial K} h
|\jump{\nabla_\Sigma v_h^l}|^2~\mbox{d}\sigma \leqslant \sum_{K \in \mathcal{T}_h} \|  h^{\frac12}
|\jump{\bfB^{-T} \bfPsh \nabla  v_h}|\|_{\Sigma_h \cap \partial
  K}^2.
\end{equation}
The right hand side may be bounded as follows
\begin{multline}\label{rhs}
 \sum_{K \in \mathcal{T}_h} \|  h^{\frac12}
|\jump{\bfB^{-T} \bfPsh \nabla  v_h}|\|_{\Sigma_h \cap \partial
  K}^2 
\\  
  \leqslant C \sum_{K \in \mathcal{T}_h} \left(\|  h^{\frac12}
|\jump{\bfB^{-T} \bfPsh} \nabla  v_h|\|_{\Sigma_h \cap \partial
  K}^2 + \|  h^{\frac12}
|\jump{\nabla  v_h}|\|_{\Sigma_h \cap \partial
  K}^2\right).
\end{multline}
For the second term in the right hand side we have by a trace
inequality from $\Sigma_h \cap \partial
  K$ to $F \in \partial K$,
\begin{equation}
 \sum_{K \in \mathcal{T}_h}  \|  h^{\frac12}
|\jump{\nabla  v_h}|\|_{\Sigma_h \cap \partial
  K}^2 \leqslant \tn v_h \tn_{\mcF}.
\end{equation}
For the first term observe that also by repeated trace inequalities,
first from $\Sigma_h \cap \partial
  K$ to $\partial K$ and then from $\partial K$ to $K$,
\begin{equation}\label{rep_trace}
\|  h^{\frac12}
|\jump{\bfB^{-T} \bfPsh} \nabla  v_h|\|_{\Sigma_h \cap \partial
  K} \leqslant C \|\jump{\bfB^{-T} \bfPsh} \|_{L^\infty(\partial K)}
h^{-\frac12} \|\nabla  v_h\|_{\Omega_h}.
\end{equation}
Now using the regularity of $\Sigma$ we may write $\jump{\bfB^{-T} \bfPsh} = \jump{\bfB^{-T} \bfPsh -
  (\bfI- b \kappa)^{-T} \bfPs}$ and consequently
\begin{multline}
\bfB^{-T} \bfPsh -
  (\bfI- b \kappa)^{-T} \bfPs = \bfB^{-T} (\bfPsh - \bfPs) \\
+  (\bfI- b
  \kappa)^{-T}( \underbrace{(\bfPs (\bfPsh - \bfPs) + \bfn \otimes (\bfn_h -
  \bfn))}_{-\delta_{\Sigma}} + \bfI)^{-T} - \bfI) \bfPs.
\end{multline}
Since for $h$ small enough the spectral radius of $\delta_{\Sigma}$ is
smaller than one there holds
\begin{equation}
(\bfI- \delta_\Sigma)^{-T} - \bfI = \left(\sum_{k=0}^\infty
\delta_\Sigma^k\right)^T - \bfI = \delta_\Sigma^T \left(\sum_{k=0}^\infty
\delta_\Sigma^k\right)^T.
\end{equation}
Therefore
\begin{equation}
\|(\bfI- \delta_\Sigma)^{-T} - \bfI)\| \leqslant \frac{ \|\delta_\Sigma\|}{1-
  \|\delta_\Sigma\|} \leqslant C h
\end{equation}
and $\|\jump{\bfB^{-T} \bfPsh} \|_{L^\infty(\partial K)} \leqslant C h$.
Using this bound together with \eqref{rhs} and \eqref{rep_trace} we may write
\begin{equation}
 \sum_{K \in \mathcal{T}_h} \|  h^{\frac12}
|\jump{\bfB^{-T} \bfPsh \nabla  v_h}|\|_{\Sigma_h \cap \partial
  K}^2 \leqslant C (h \|\nabla v_h\|^2_{\Omega_h} + \tn v_h \tn^2_{\mcF}).
\end{equation}
The bound \eqref{cont2} then follows using the
arguments of Lemma 4.2 of \cite{BHL13} (see also Lemma 5.3 of
\cite{BHLM15}) leading to
\begin{equation}
h \|\nabla v_h\|^2_{\Omega_h}  \leqslant C (h^2 \|\nabla_{\Sigma_h}
v_h\|^2_{\Sigma_h} + \tn v_h \tn_{\mcF}^2)
\end{equation}
and the norm equivalence $\|\nabla_\Sigma v^l_h\|_{\Sigma} \sim \|\nabla_{\Sigma_h} v_h\|_{\Sigma_h}$.
\end{proof}
We will first prove some conformity error bounds that we
collect in a lemma.
\begin{lem}\label{conf_error}
Let $u_h$ be the solution of \eqref{full_FEM} and assume that
$hk < 1$.
Then 
\begin{gather}
|a_h(u_h, v_h) - a(u^l_h,v^l_h)| \leqslant C  h^2 \tn
u_h^l\|_{\Sigma,k} \tn v_h^l\tn_{\Sigma,k}
\\
|l_s(v^l_h) - l_h(v_h)| \leqslant C_f (h^2 \tn v^l_h\tn_{\Sigma,k} + h\tn v_h\tn_{\mcF})
\end{gather}
and 
\begin{multline}\label{stab_conf}
|s_h(u_h, v_h) - s(u^l_h,v^l_h)| \leqslant C (\tn u^l_h\tn_{GLS} +\tn u_h\tn_{\mcF} ) \tn v_h\tn_{\mcF}\\
+  C \tn v^l_h\tn_{GLS} \tn u_h\tn_{\mcF}
+Ch^2\tn
u_h^l \|_{\Sigma,k} \tn v_h^l\tn_{\Sigma,k}.
\end{multline}
\end{lem}
\begin{proof}
For the first term we observe that
\begin{align}
|a_h(u_h, v_h) - a(u^l_h,
  v^l_h)| 
  &\leqslant  |(\nabla_{\Sigma_h} u_h, \nabla_{\Sigma_h}
  v_h)_{\Sigma_h} - (\nabla_{\Sigma} u^l_h, \nabla_{\Sigma}
  v_h^l)_{\Sigma}  |
  \\ \nonumber
&\qquad  + | (k^2 u_h, v_h)_{\Sigma_h} - (k^2 u_h,  v_h)_{ \Sigma}| 
\\
&\leqslant C h^2 \|\nabla_{\Sigma} u^l_h\|_{\Sigma}  \|\nabla_{\Sigma}
v^l_h\|_{\Sigma} + \int_{\Sigma_h} k^2 u_h \bar v_h (1 - |
\bfB|) \mbox{d}\sigma_h
\end{align}
where we used the result on the Laplace-Beltrami part from \cite{BHL13}.
For the zero order term term we observe that by \eqref{detBbounds}
\begin{equation}\label{low_ord_term}
|\int_{\Sigma_h} k^2 u_h \bar v_h (1 - | \bfB|)
  \mbox{d} \sigma_h| \leqslant C h^2 \| k u_h^l\|_{\Sigma} \| k v^l_h\|_{\Sigma}.
\end{equation}
For the control of the conformity error of the right hand side we observe that
\begin{equation}
l_s(v^l_h) - l_h(v_h) = \int_{\Sigma_h} f_h \bar v_h(1 -
\gamma_s h^2 k^2) (|\bfB|-1) ~\mbox{d} \sigma_h 
- \sum_{K \in
  \mcK_h} \int_{\Sigma \cap K} f
\gamma_s h^2 \Delta_{\Sigma} \bar  v^l_h ~ \mbox{d} \sigma.
\end{equation}
The first term on the right hand side was bounded in
\cite{BHL13},
\begin{equation}
\int_{\Sigma_h} f_e \bar v_h(1 -
\gamma_s h^2 k^2) (|\bfB|-1) ~\mbox{d} \sigma_h \leqslant C_f h^2 \|v_h\|_{\Sigma}.
\end{equation}
Once again we use the relation \eqref{eq:Delta_bound}
and by changing the domain of integration and applying
Lemma \ref{grad_cont} we obtain
\begin{equation}\label{LapBelcont}
\sum_{K \in \mcK_h} h^2 \|\Delta_{\Sigma} v_h^l\|^2_{\Sigma
  \cap K} \leqslant C_{\kappa} h^2 \|\nabla v_h^l\|^2_{\Sigma}\\
\leqslant C_{\kappa} (h^2  \|\nabla_{\Sigma} v_h^l\|^2_{\Sigma} +\tn v_h\tn^2_{\mcF}).
\end{equation}
 Hence the second term may be bounded as
\begin{equation}
\sum_{K \in
  \mcK_h} \int_{\Sigma \cap K} f
\gamma_s h^2 \Delta_{\Sigma} \bar v_h ~ \mbox{d} \sigma \leqslant C_{\kappa} h
\|f\|_{\Sigma} (h \|\nabla_{\Sigma} v_h^l\|_{\Sigma} + \tn v_h\tn_{\mcF}).
\end{equation}
For the Galerkin least squares term we may write
\begin{align}
s_h(u_h, v_h) - s(u^l_h,v^l_h) 
&= (h^2 k^2 u_h, k^2 v_h)_{\Sigma_h} -  (h^2
k^2 u^l_h,k^2 v^l_h)_{\Sigma} 
\\  \nonumber
&\qquad
-\sum_{K \in \mcK_h} (h^2 \Delta_\Sigma u^l_h, \Delta_\Sigma
v^l_h)_{\Sigma\cap K} 
\\ \nonumber
&\qquad 
+ \sum_{K \in \mcK_h} (h^2
\Delta_\Sigma u^l_h,\Delta_\Sigma v^l_h + k^2 v_h^l)_{\Sigma\cap K}
\\ \nonumber
&\qquad
+\sum_{K \in \mcK_h} (h^2(
\Delta_\Sigma u^l_h + k^2 u_h^l), \Delta_\Sigma v^l_h)_{\Sigma\cap K}.
\end{align}
Using the bounds \eqref{low_ord_term} and \eqref{LapBelcont} we have
\begin{gather}
 (h^2 k^2 u_h, k^2 v_h)_{\Sigma_h} -  (h^2
k^2 u^l_h,k^2 v^l_h)_{\Sigma} \leqslant C h^2 (hk)^2 \|k u^l_h\|_{\Sigma} \|k v^l_h\|_{\Sigma}
\\
\sum_{K \in \mcK_h} (h^2 \Delta_\Sigma u^l_h, \Delta_\Sigma
v^l_h)_{\Sigma \cap K}  \leqslant   (h \|\nabla_{\Sigma} u_h^l\|_{\Sigma} +\tn u_h\tn_{\mcF} )(h \|\nabla_{\Sigma} v_h^l\|_{\Sigma} + \tn v_h\tn_{\mcF})
\end{gather}
and
\begin{equation}
\sum_{K \in \mcK_h} (h^2
\Delta_\Sigma u^l_h,\Delta_\Sigma v^l_h + k^2 v_h^l)_{\Sigma\cap K}
\leqslant C  (h \|\nabla_{\Sigma} u_h^l\|_{\Sigma} +\tn u_h\tn_{\mcF})
\tn v_h^l\tn_{GLS}.
\end{equation}
\end{proof}
An immediate consequence of the previous result is the following bounds
on the conformity error of the form $A_h(\cdot,\cdot)$.
\begin{cor}
Let $u_h$ be the solution of \eqref{full_FEM} and assume that
$hk < 1$.
Then for all $\epsilon > 0$
\begin{multline}\label{Abound2}
|A_h(v_h,w_h) - A(v_h^l, w_h^l)| \leqslant C (1+ |\gamma_s|) h^2 |\tn v^l_h \tn_{\Sigma,k} \tn
w^l_h \tn_{\Sigma,k} 
\\
+ C |\gamma_s| \epsilon ( \tn v^l_h \tn_{GLS}^2 + \tn
w^l_h \tn_{GLS}^2)  + C \left( \frac{|\gamma_s|}{4 \epsilon} + 1 \right)
\left(\tn v_h \tn_{\mcF}^2 + \tn
w_h \tn_{\mcF}^2\right).
\end{multline}
\end{cor}
\begin{proof} 
Follows directly from the previous lemma, and an arithmetic-geometric inequality.
\end{proof}
The proof of convergence of the fully discrete scheme now follows the
same model as that of the semi-discrete scheme, estimating this time
also the error induced by integrating the equations on the discrete
representation of the surface.
\begin{lem}\label{stab_convergence}
Assume that $\mbox{Im}(\gamma_j) > C_{\Sigma,1} |\gamma_s|^2 Im(\gamma_s)^{-1}  + C_{\Sigma,2}
$, where the
constants $C_{\Sigma,1} $ and $C_{\Sigma,2}$ only depends on the smoothness of the surface. Then,
\begin{equation}
\tn \pi_h^l u - u^l_h \tn_{GLS}^2 + \tn  \pi_h u^e
- u_h \tn_{\mcF} \leqslant  C_{f,\gamma} (hk) + C_{\gamma} h\tn \pi_h^l u - u^l_h \tn_{\Sigma,k}.
\end{equation}
\end{lem}
\begin{proof}
  Using the short-hand notation $\pi_h^l u := ((\pi_h u^e)|_{\Sigma_h})^l$, we define the discrete error
  on $\Sigma_h$ and its corresponding lift to $\Sigma$ by
   $\xi_h:=\pi_h u^e
- u_h$ and $\xi_h^l:=\pi_h^l u - u^l_h$, respectively.

Using the definition of the scheme on the exact and the discrete surfaces
we may write
\begin{align}
&\mbox{Im}(\gamma_s) \tn \xi_h^l \tn_{GLS}^2 + \mbox{Im}(\gamma_j) \tn  \xi_h \tn_{\mcF}^2
\\
&\qquad = \mbox{Im}[a(\xi_h^l ,\xi_h^l ) + \gamma_s
s(\xi_h^l ,\xi_h^l) 
+ \gamma_j j(\xi_h ,\xi_h)] 
\\
&\qquad = \mbox{Im}[a(\pi_h^l u - u ,\xi_h^l ) + \gamma_s
s(\pi_h^l u - u ,\xi_h^l) + \gamma_j j(\pi_h u^e ,\xi_h)
\\ \nonumber
&\qquad \qquad 
+ l_s(\xi_h^l) - l_h(\xi_h) 
+ A_h(u_h ,\xi_h ) - A(u^l_h ,\xi_h^l ) ].
\end{align}
For the first three terms in the right hand side we use Lemma
\ref{cont_lem2} together with similar
arguments as for \eqref{error_rep} to obtain the bound
\begin{multline}\label{cont_full_disc}
|a(\pi_h^l u - u ,\xi_h^l ) + \gamma
s(\pi_h^l u - u ,\xi_h^l) + \gamma j(\pi_h u^e ,\xi_h)| \\  \leqslant  C_{f,\gamma}^2   (h k)^2 + h^2
\tn \xi_h^l\tn_{\Sigma,k}^2 + \frac14 (\mbox{Im}(\gamma_s) \tn \xi_h^l \tn_{GLS}^2 + \mbox{Im}(\gamma_j) \tn  \xi_h\tn_{\mcF}^2).
\end{multline}
For the remaining terms we use the result of Lemma \ref{conf_error} to
deduce 
\begin{align}
 l_s(\xi_h^l) - l_h(\xi_h)  
 &\leqslant  C_f (h^2 \tn \xi_h^l \tn_{\Sigma,k} + h\tn \xi_h \tn_{\mcF}) 
\\
&\leq
 C_{f,\gamma_s}^2 (h k)^2
 + h^2 k^{-2} \tn \xi_h^l \tn_{\Sigma,k}^2 + \frac14 \tn \xi_h \tn_{\mcF}^2.
\end{align}
To bound the conformity error of $A_h(\cdot,\cdot)$ it is convenient
to start from \eqref{Abound2} and write
\begin{multline}
A_h(u_h ,\xi_h^l ) - A(u^l_h ,\xi_h^l ) \leqslant C h^2 (1 + |\gamma_s|) |\tn u^l_h \tn_{\Sigma,k} \tn
\xi_h^l \tn_{\Sigma,k} 
\\
+ C |\gamma_s| \epsilon ( \tn u^l_h \tn_{GLS}^2 + \tn
\xi_h^l \tn_{GLS}^2)  + C \left( \frac{|\gamma_s|}{4 \epsilon} + 1 \right)
\left(\tn u_h \tn_{\mcF}^2 + \tn
\xi_h \tn_{\mcF}^2\right).
\end{multline}
By adding and subtracting $\pi_h^l u$ in the norms on $u_h^l$ and $\pi_h
u^e$ in the norms on $u_h$,
applying the triangular inequality and applying the bounds
\eqref{pihlbound} in combination with \eqref{regularity}  we may rewrite this as
\begin{multline}
A_h(u_h ,\xi_h^l ) - A(u^l_h ,\xi_h^l ) \leq
C_{\gamma_s,\epsilon}  h^2 (k^2 \|f\|_{\Sigma}^2 + \tn
\xi_h^l \tn_{\Sigma,k}^2) \\+ C |\gamma_s| \epsilon \tn
\xi_h^l \tn_{GLS}^2  + C \left( \frac{|\gamma_s|}{4
    \epsilon} + 1 \right) \tn
\xi_h\tn_{\mcF}^2.
\end{multline}
Choosing $\epsilon = Im(\gamma_s)  (2 C |\gamma_s|)^{-1}$ and fixing
$\gamma_j$ such that $Im(\gamma_j) > C \left( \frac{|\gamma_s|}{4
    \epsilon} + 1 \right)+1$ there exists constants $C_\gamma, \,C_{f,\gamma} >0$ such that
\begin{equation}
C_\gamma (\tn \xi_h^l \tn_{GLS}^2 + \tn  \xi_h\tn_{\mcF}^2)  \leqslant  C_{f,\gamma}^2 (hk)^2 + C_{\gamma} h^2 \tn \xi_h^l \tn_{\Sigma,k}^2.
\end{equation}
\end{proof}
\begin{lem}\label{L2conf_err}
For the error in the $L^2$-norm there holds
\begin{equation}
\|u - u_h^l\|_{\Sigma} \leqslant C_{f,\gamma} (hk)^2 + C h^2 k \tn
\pi_h^l u - u_h^l\tn_{\Sigma,k}.
\end{equation}
\end{lem}
\begin{proof}
We let $z$ be the solution of \eqref{weak_Helmholtz} with right hand
side $f =  u - u_h^l$. It follows that
\begin{equation}
\|u - u_h^l\|^2_{\Sigma}  = a( u - u_h^l, z - \pi^l_h
z) + a(u - u_h^l, \pi^l_h z) = I +II.
\end{equation}
By the continuity of $a(\cdot,\cdot)$ (Lemma \ref{cont_lem2}), the approximation properties of
$\pi^l_h z $ and the regularity estimate \eqref{regularity}  we have for the first term
\begin{equation}\label{boundI}
I \leqslant C h k  (\tn u - u_h^l \tn_{GLS} + \tn u_h \tn_{\mcF}) \|u - u_h^l\|_{\Sigma}.
\end{equation}
Using the definition of the finite element methods \eqref{semi_FEM}
and \eqref{full_FEM} we have for the second term
\begin{multline}
II = l_s(\pi^l_h z) - l_h(\pi_h z^e) + A_h(u_h,\pi_h z^e) -
A(u^l_h,\pi^l_h z) \\
- \gamma_s s(u - u_h^l, \pi^l_h z) + \gamma_j j( u_h, \pi_h z^e).
 \end{multline}
Using Lemma \ref{conf_error} in the two first terms and the
Cauchy-Schwarz intequality in the two last we have
\begin{align}
II &\leqslant C_{\gamma_s} \Big( h^2  \tn\pi^l_h z\tn_{\Sigma,k} + h \tn \pi_h
z^e\tn_{\mcF} + h^2 \tn u^l_h\tn_{\Sigma,k} \tn
\pi^l_h
z\tn_{\Sigma,k} 
\\ \nonumber 
&\qquad \qquad + (\tn u^l_h\tn_{GLS} +\tn u_h\tn_{\mcF} ) \tn
\pi_h
z^e\tn_{\mcF} 
+  \tn u_h\tn_{\mcF} \tn \pi^l_h
z\tn_{GLS}\Big)
\\
&\qquad + |\gamma_s| \tn u - u_h^l\tn_{GLS}\tn \pi^l_h z \tn_{GLS} \|
+|\gamma_j|\tn u^e -
u_h\tn_{\mcF} \tn \pi_h z^e \tn_{\mcF}.
\end{align}
Recalling the equations \eqref{stab_dual1} and \eqref{stab_dual2} and
\eqref{pihlbound} we
have the bounds
\begin{equation}
\tn \pi^l_h z \tn_{GLS} + \tn \pi_h z^e \tn_{\mcF} \leqslant C (hk)
\|u - u^l_h\|_{\Sigma}, \quad  \tn \pi^l_h z\tn_{\Sigma,k} \leq
C \|u - u^l_h\|_{\Sigma}
\end{equation}
and similarly
\begin{equation}
\tn \pi^l_h u \tn_{GLS} + \tn \pi_h u^e \tn_{\mcF} \leqslant C_f
(hk).
\end{equation}
Adding and subtracting $\pi_h^l u$ in all the norms on $u_h^l$ and
using a triangle inequality and the above bounds on norms of  $\pi^l_h z$ and
$\pi_h^l u$ we arrive at the bound
\begin{equation}\label{boundII}
II \leqslant C_{\gamma_s} ((hk)^2 + (h^2 k)   \|\pi_h^l u - u_h\|_{\Sigma}
+ hk (\tn \pi^l_h u - u_h^l\tn_{GLS} + \tn \pi_h u^e - u_h\tn_{\mcF}).
\end{equation}
By summing up the bounds \eqref{boundI} and \eqref{boundII} we arrive
at the inequality
\begin{align}
\|u - u_h^l\|_{\Sigma} &\leqslant C\Big( (hk)^2 +  h k (\tn \pi_h^l u - u_h^l \tn_{GLS} + \tn \pi_h u^e - u_h \tn_{\mcF})  
\\  \nonumber
&\qquad 
+  h^2k \tn
\pi_h^l u - u_h^l \tn_{\Sigma,k}\Big).
\end{align}
Using the result of Lemma \ref{stab_convergence} the conclusion follows.
\end{proof}
We now use the above lemmas for the fully discrete formulation to
prove our main result, an  apriori error estimate in the $\tn\cdot
\tn_{\Sigma,k}$-norm. This result may then be used to prove stability
of the discrete solution under the condition $hk$ small, similarly as
in Lemma \ref{apriori_bound}. We leave the details to the reader.
\begin{thm}
Let $u$ be the solution of \eqref{weak_Helmholtz} satisfying the
estimate \eqref{regularity} and let $u_h$ be the solution of
\eqref{full_FEM}.
Then for $hk$ sufficiently small
\begin{equation}
\tn u - u_h^l\tn_{\Sigma,k}  \leqslant C_{f,\gamma} hk ( 1 + hk^2).
\end{equation}
\end{thm}
\begin{proof}
First we observe that by the triangle inequality there holds
\begin{equation}
\tn u - u_h^l\tn_{\Sigma,k} \leqslant \tn \pi_h^l u - u_h^l\tn_{\Sigma,k}+\tn u - \pi_h^l u\tn_{\Sigma,k}.
\end{equation}
Since the bound was proven for the second term in the right
hand side in Lemma \ref{lem:approx} we only need to consider 
the first term. Once again we use the notation $\xi_h^l := \pi_h^l u -
u_h^l$ and $\xi_h := \pi_h u^e - u_h$.

It follows by the definition of $A(\cdot,\cdot)$ and the assumption
that $Re[\gamma_s]=Re[\gamma_j]=0$ that
\begin{equation}\label{sigmakbound}
\tn \xi_h^l\tn_{\Sigma,k}^2 = 
\mbox{Re}(A(\xi_h^l,\xi_h^l)) + 2 \|k(\xi_h^l)\|_\Sigma^2
\leqslant  |A(\xi_h^l,\xi_h^l)|+ 2 \|k \xi_h^l\|_\Sigma^2
\end{equation}
and
\begin{equation}
A(\xi_h^l,\xi_h^l) =  A(\pi_h^l u - u,\xi_h^l) + l_s(\xi_h^l) - l_h(\xi_h) \\
+ A_h(u_h, \xi_h) - A(u_h^l,\xi_h^l)+ \gamma_j
j(u_h,\xi_h).
\end{equation}
Using the result of \eqref{cont_full_disc} we have
\begin{equation}\label{Asemicont}
A(\pi_h^l u - u,\xi_h^l) \leqslant C_{f,\gamma_s}^2   (h k)^2 + h^2
\tn \xi_h^l\tn_{\Sigma,k}^2 \\
+ \frac14 (\mbox{Im}(\gamma_s) \tn \xi_h^l \tn_{GLS}^2 + \mbox{Im}(\gamma_j) \tn  \xi_h\tn_{\mcF}^2).
\end{equation}
Recalling the second bound of Lemma \ref{conf_error}  we also have
\begin{multline}\label{lllbound}
 l_s(\xi_h^l) - l_h(\xi_h) \leqslant C_f (hk)^{2}  + C h^2k^{-1}  \tn\pi_h^l u
 -u_h^l\tn_{\Sigma,k}^2 +   h \tn
\xi_h\tn^2_{\mcF} \\
\leqslant  C_f (hk)^{2}  + C h^2 \tn\pi_h^l u
 -u_h^l\tn_{\Sigma,k}^2.
\end{multline}
Finally, using \eqref{Abound2} and after adding and subtracting $\pi_h^l u$ and applying the
triangular inequality and the result of Lemma \ref{stab_convergence}
we have
\begin{align}\label{AAbound}
A_h(u_h, \xi_h) - A(u_h^l,\xi_h^l) &\leqslant C h^2 \tn
\xi_h^l \tn^2_{\Sigma,k} + C \tn
\xi_h^l \tn_{GLS}^2 + C\tn
\xi_h \tn_{\mcF}^2
\\ \nonumber
&\qquad + C h^2 (\tn \pi_h^l u \tn^2_{\Sigma,k} 
+ \tn \pi_h^l u \tn_{GLS}^2) 
\\
& \leqslant C_f (hk)^2 + C h^2 \tn
\xi_h^l\tn^2_{\Sigma,k}.
\end{align}
Applying the results of \eqref{Asemicont}, \eqref{lllbound},
\eqref{AAbound}, and the $L^2$-error estimate of Lemma \ref{L2conf_err}
in \eqref{sigmakbound} we obtain
\begin{equation}
\tn \pi_h^l u - u_h^l\tn^2_{\Sigma,k} \leqslant C_f (1+h^2 k^4) (hk)^2 + C h^2(1+k^2) \tn\pi_h^l u
 -u_h^l\tn_{\Sigma,k}^2.
\end{equation}
Since $hk$ is assumed to be small, so that $Ch^2(1+k^2)<1$, the last term in the 
right hand side can be absorbed in the left hand side and the proof is complete.
\end{proof}
\section{Numerical Examples}

In the numerical examples below, the $L_2$ errors on the exact surface are approximated 
by the corresponding expression on the discrete surface,
\begin{equation}
\|u - u_h^l\|_{\Sigma} \approx \|u^e - u_h\|_{\Sigma_h}.
\end{equation}

\subsection{Varying Wave Number}
We consider the sphere with radius $r=1/2$ and the following stabilization parameters: $\gamma_s=i$, $\gamma_j= 10^{-3}i$, with $i$ the imaginary unit.
We use a fabricated solution 
\begin{equation}\label{sol}
u=(x-1/2)(y-1/2)(z-1/2)
\end{equation}
and construct the right-hand side accordingly. In Fig. \ref{fig:sphere} we show a typical discretization and corresponding approximate solution. In Fig. \ref{fig:sphereconv} we 
show the convergence patterns for different wave numbers and note that the rate is unaffected. 

\subsection{Varying Geometry}
In this example, we consider the spheroid with one main axis having length $\text{R}_\text{max}=1/2$ constant and the other with length $\text{R}_\text{min}$ varying.
The data are the same as in the previous example but with constant wave number $k^2=1$. In Fig. \ref{fig:spheroids} we show two different spheroids and in Fig. \ref{fig:cubeconv} we show the convergence which is optimal independent of geometry. Finally, in Fig. \ref{fig:cube}, we consider a more demanding geometry, defined as the zero isoline of
\[
\phi = (x^2+y^2-4)^2+(z^2-2)^2+(y^2+z^2-4)^2+(x^2-1)^2+(z^2+x^2-4)^2+(y^2-1)^2-15
\]
and in Fig. \ref{fig:spheroideconv} the corresponding observed
convergence using the same parameters as for the spheroids. Similarly
as in the previous example we here observe that the rate is unaffected
by the geometry.
\subsection{Stability Close to Eigenvalues}\label{sec:stab_close_eigenval}
To illustrate the enhanced stability of the stabilized method, we consider the unit sphere (of radius 1). On this sphere, the non--zero eigenvalues of the Laplace--Beltrami operator can be analytically computed as
$\lambda = m (m+1)$, $m=1,2,\ldots$ \cite{Shu01}. We consider again the exact solution (\ref{sol}) and compute the $L_2$ error on a fixed mesh under varying $k^2$ close to the lowest eigenvalue. In Fig. \ref{fig:eigen1} we show how the error behaves using the same stabilization parameters as above.
In Fig. \ref{fig:eigen2} we give a close-up of the error closer to the
eigenvalue, and in Fig. \ref{fig:eigen3} we give the corresponding
errors without stabilization. Note that further closeups would result
in further increases of the error for the unstabilized approximation. With stabilization, the $L_2$ error increases but remains bounded as we pass the eigenvalue, unlike the case where no stabilization is added. Note that resonance occurs, in the unstabilized method, for a $k^2$--value slightly higher than $k^2=2$, which is to be expected in a conforming Galerkin finite element method (cf., e.g, \cite{StrFix73}).

\paragraph{Acknowledgment} This research was supported by: EPSRC,
Award No. EP/J002313/1 and EP/P01576X/1, (EB);  The Foundation for Strategic Research, Grant No.\ AM13-0029 (PH, ML, 
and AM); and the Swedish Research Council, Grants Nos.\  2013-4708, 2017-03911, (ML) and
2017-05038 (AM).

\bibliographystyle{abbrv}
\bibliography{LapBel}
\newpage
\begin{figure}
\centering
\includegraphics[width=0.7\linewidth]{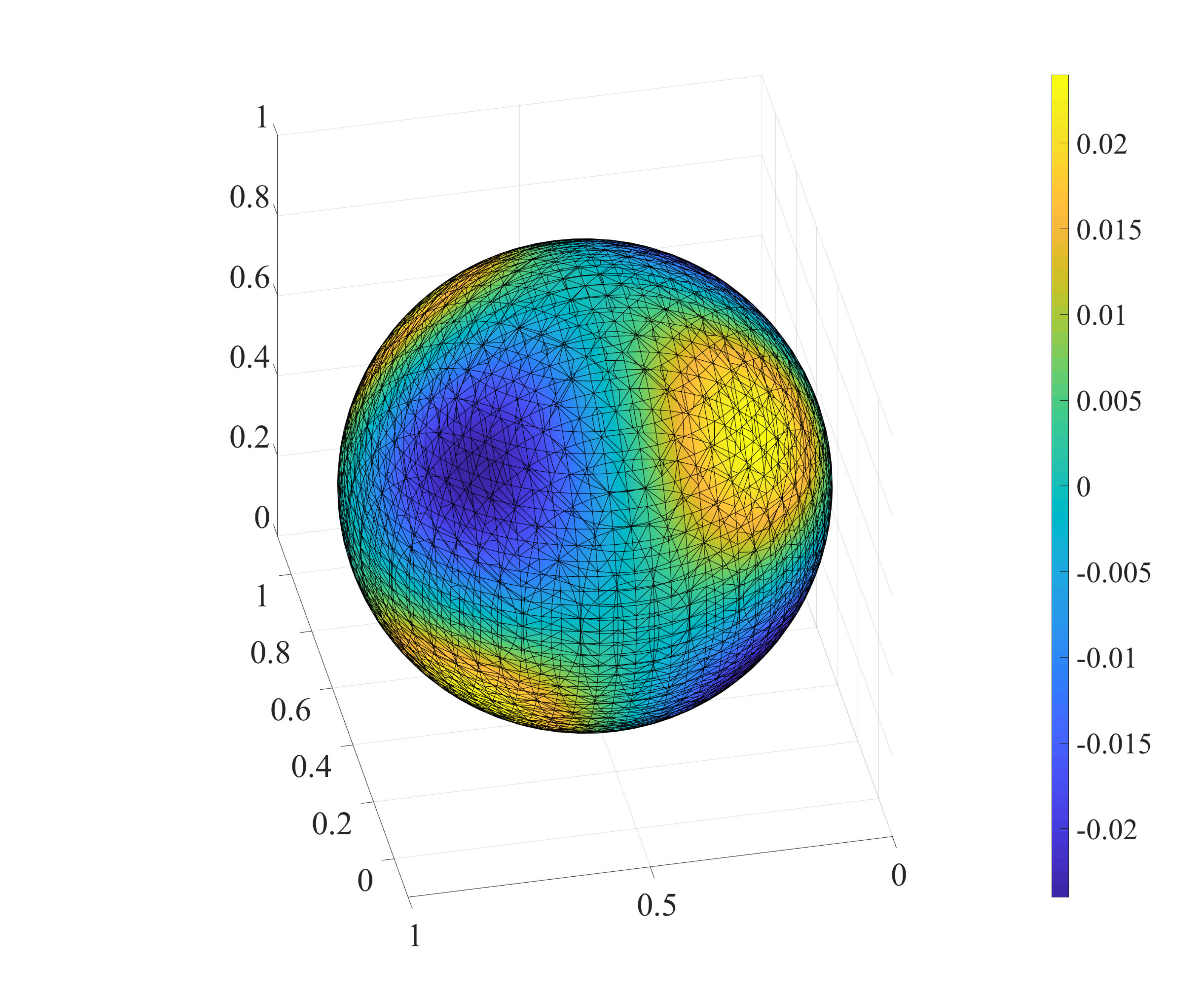}
\caption{A discretization of the sphere with corresponding discrete solution.}\label{fig:sphere}
\end{figure}
\begin{figure}
\centering
\includegraphics[width=0.7\linewidth]{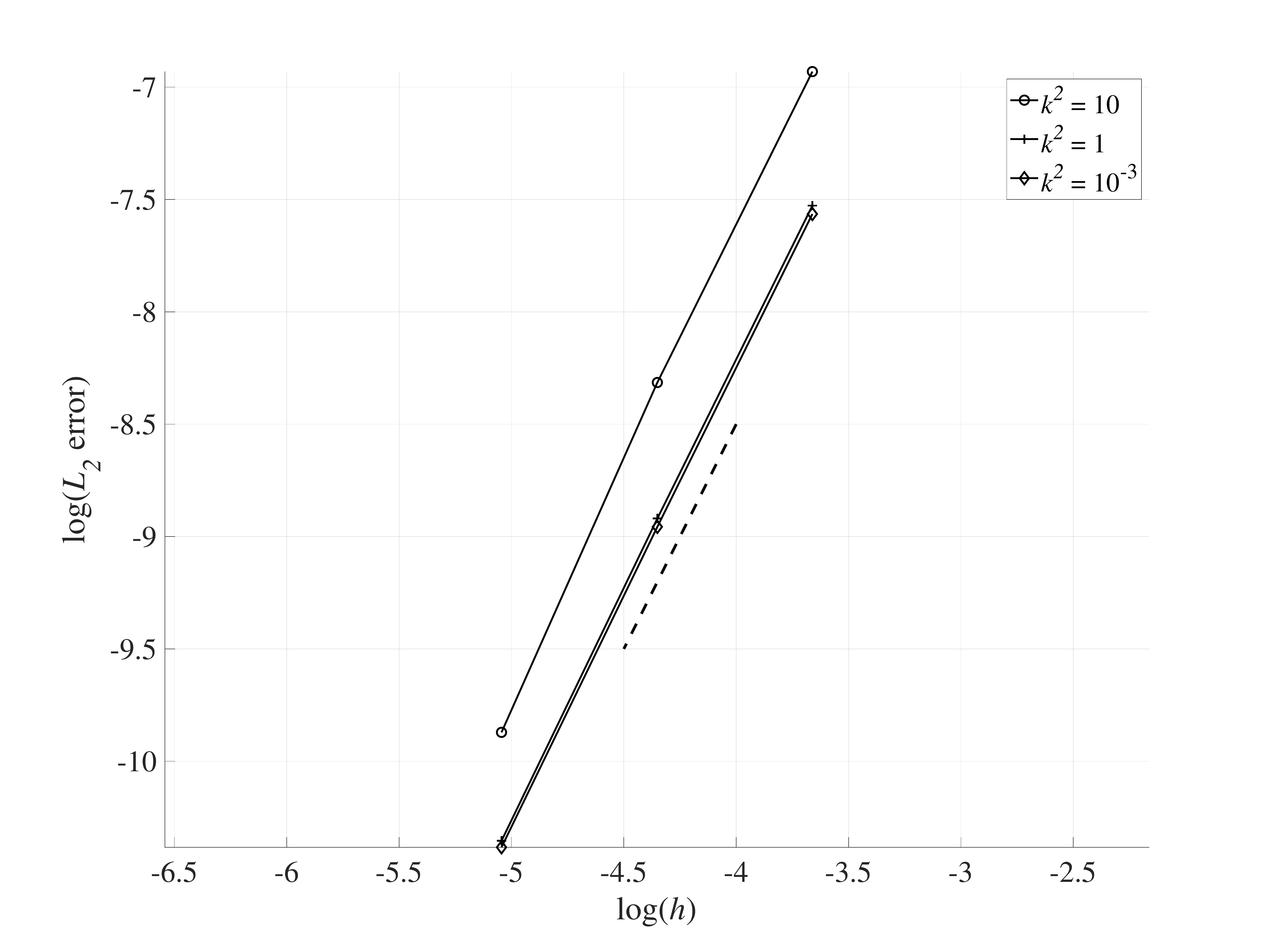}
\caption{Convergence for different wave numbers. Dotted line has inclination 2:1.}\label{fig:sphereconv}
\end{figure}
\begin{figure}
\centering
\includegraphics[width=0.7\linewidth]{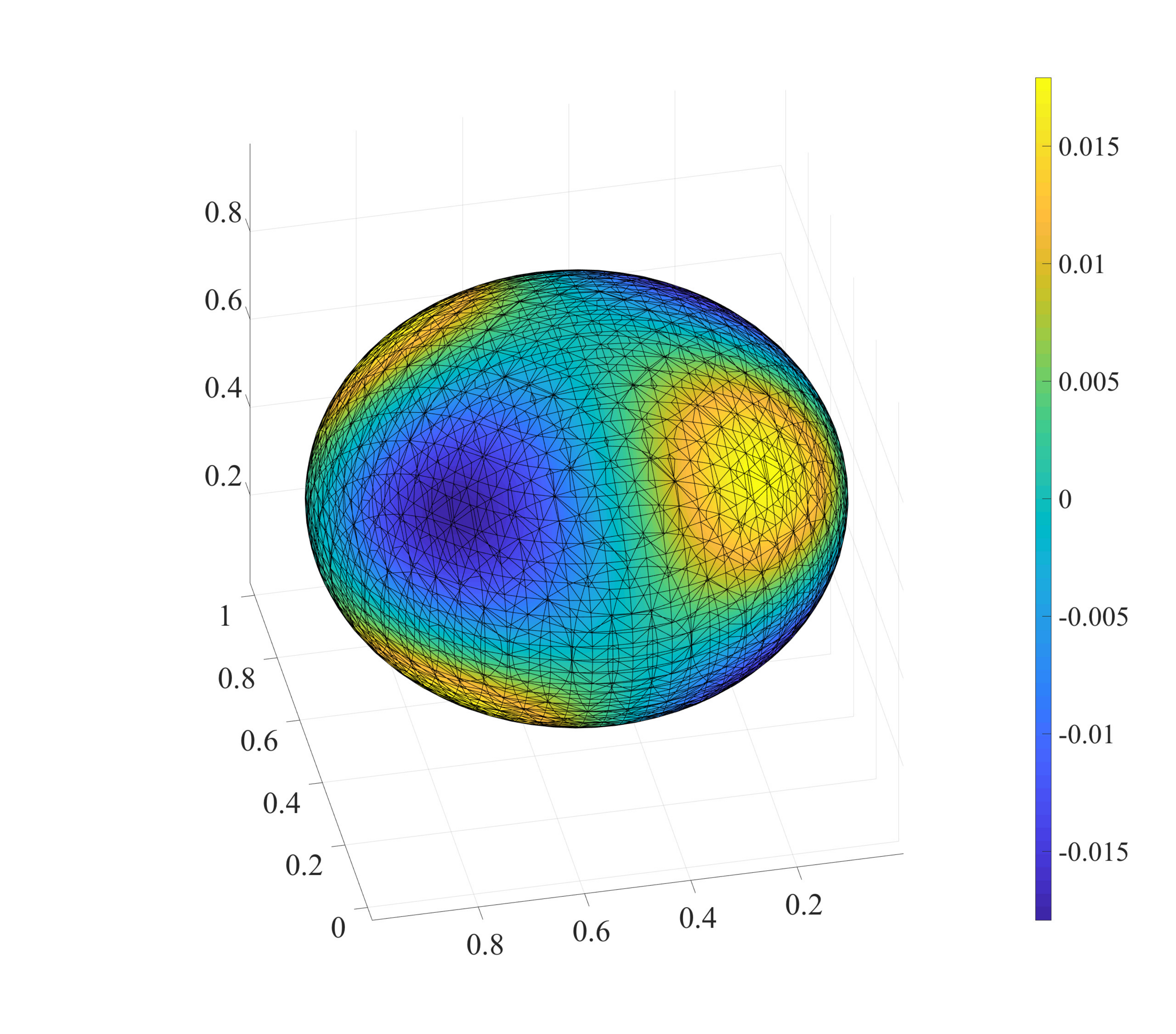}
\includegraphics[width=0.7\linewidth]{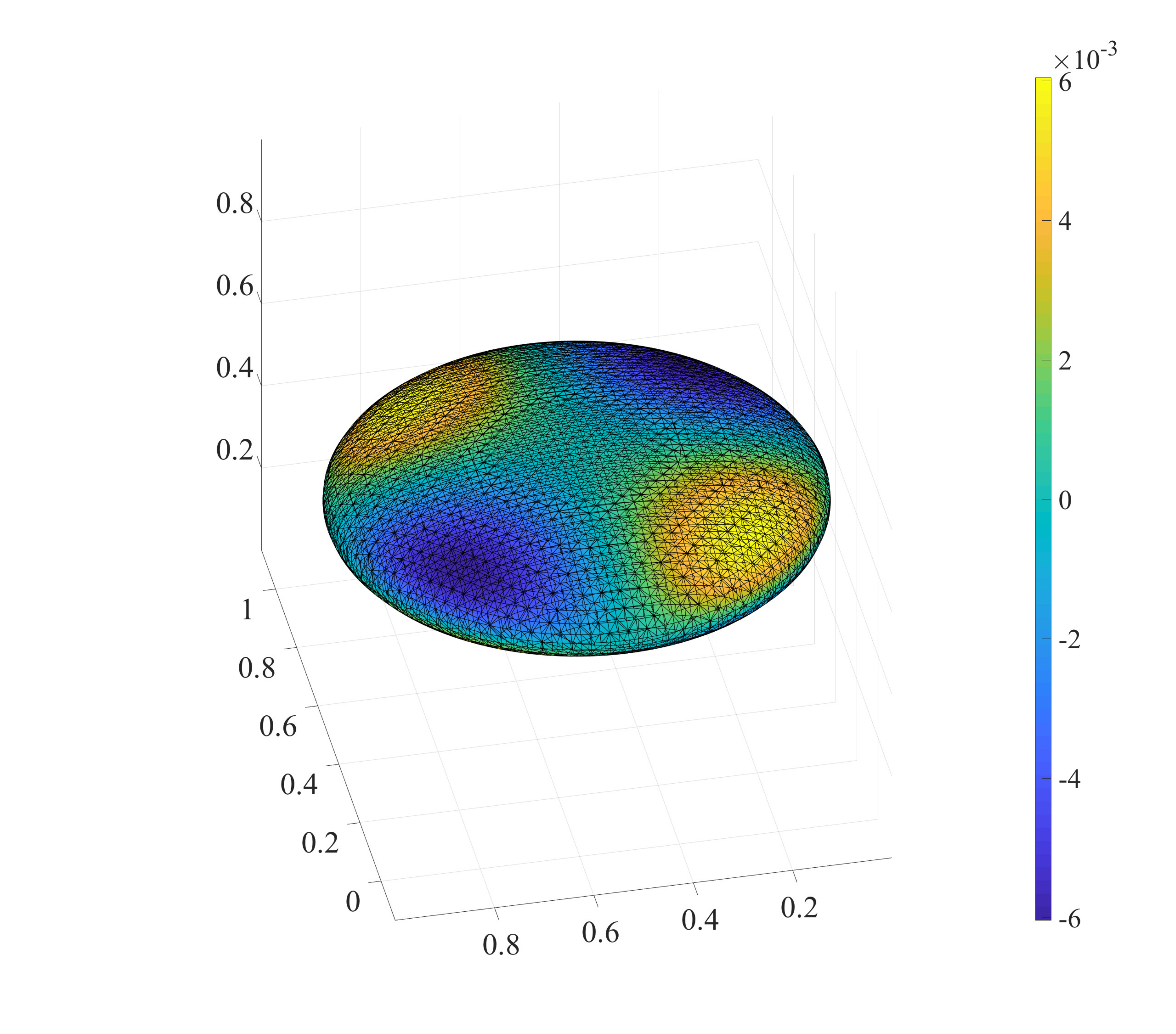}
\caption{Discretization of spheroids with corresponding discrete solutions.}\label{fig:spheroids}
\end{figure}

\begin{figure}
\centering
\includegraphics[width=0.7\linewidth]{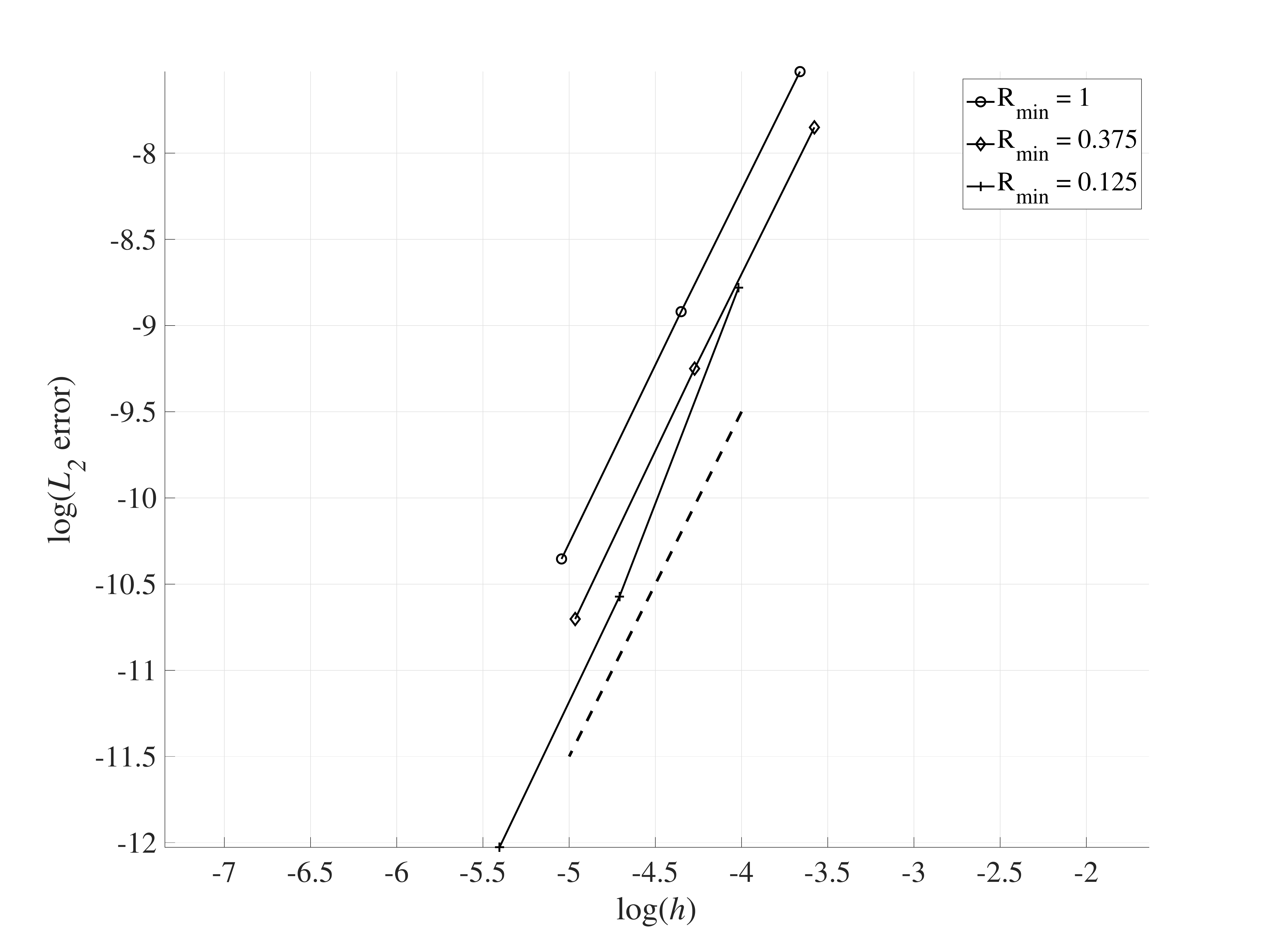}
\caption{Convergence for different spheroid geometries. Dotted line has inclination 2:1.}\label{fig:spheroideconv}
\end{figure}

\begin{figure}
\centering
\includegraphics[width=0.7\linewidth]{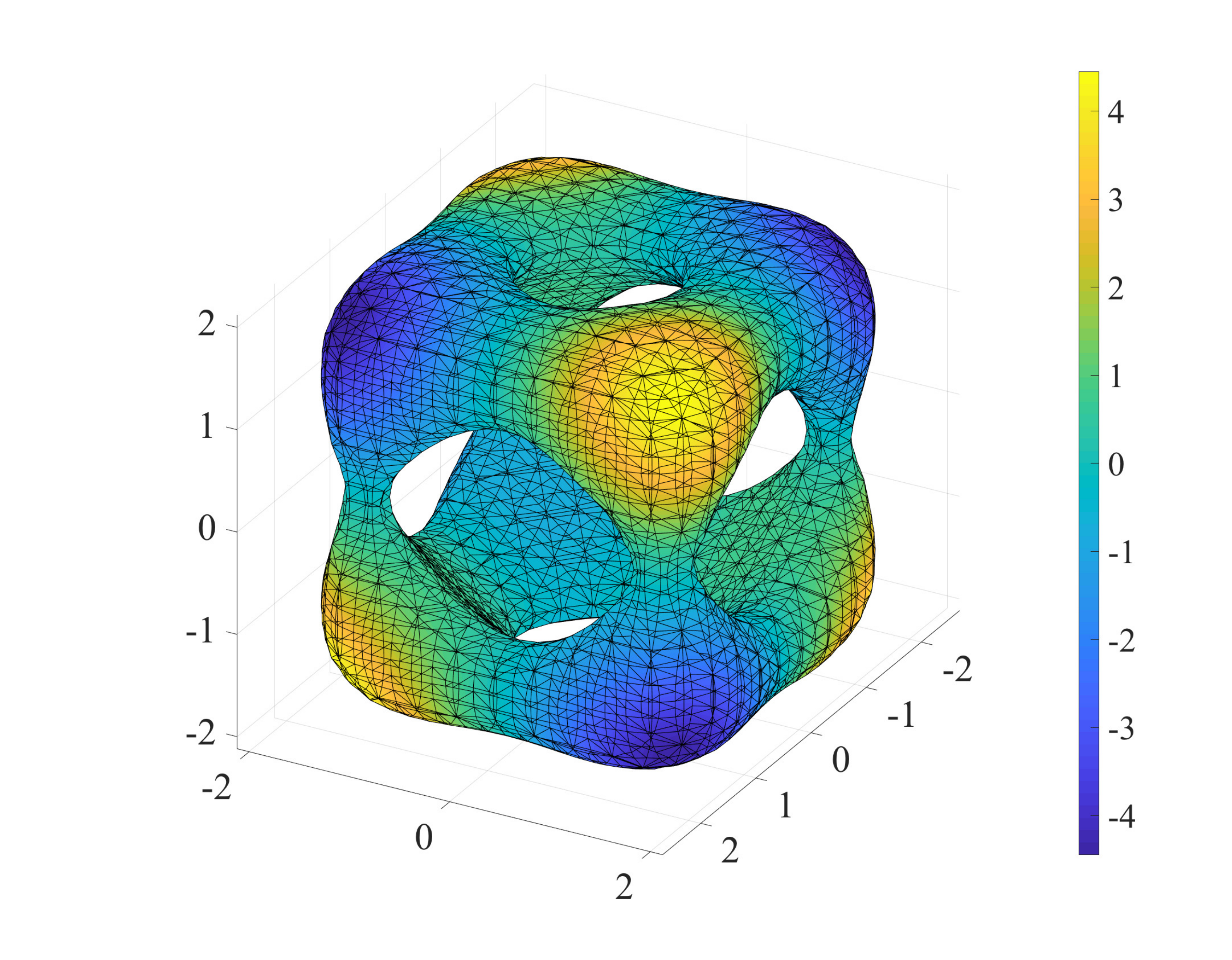}
\caption{Discretization of a more complex geometry with corresponding discrete solution.}\label{fig:cube}
\end{figure}

\begin{figure}
\centering
\includegraphics[width=0.7\linewidth]{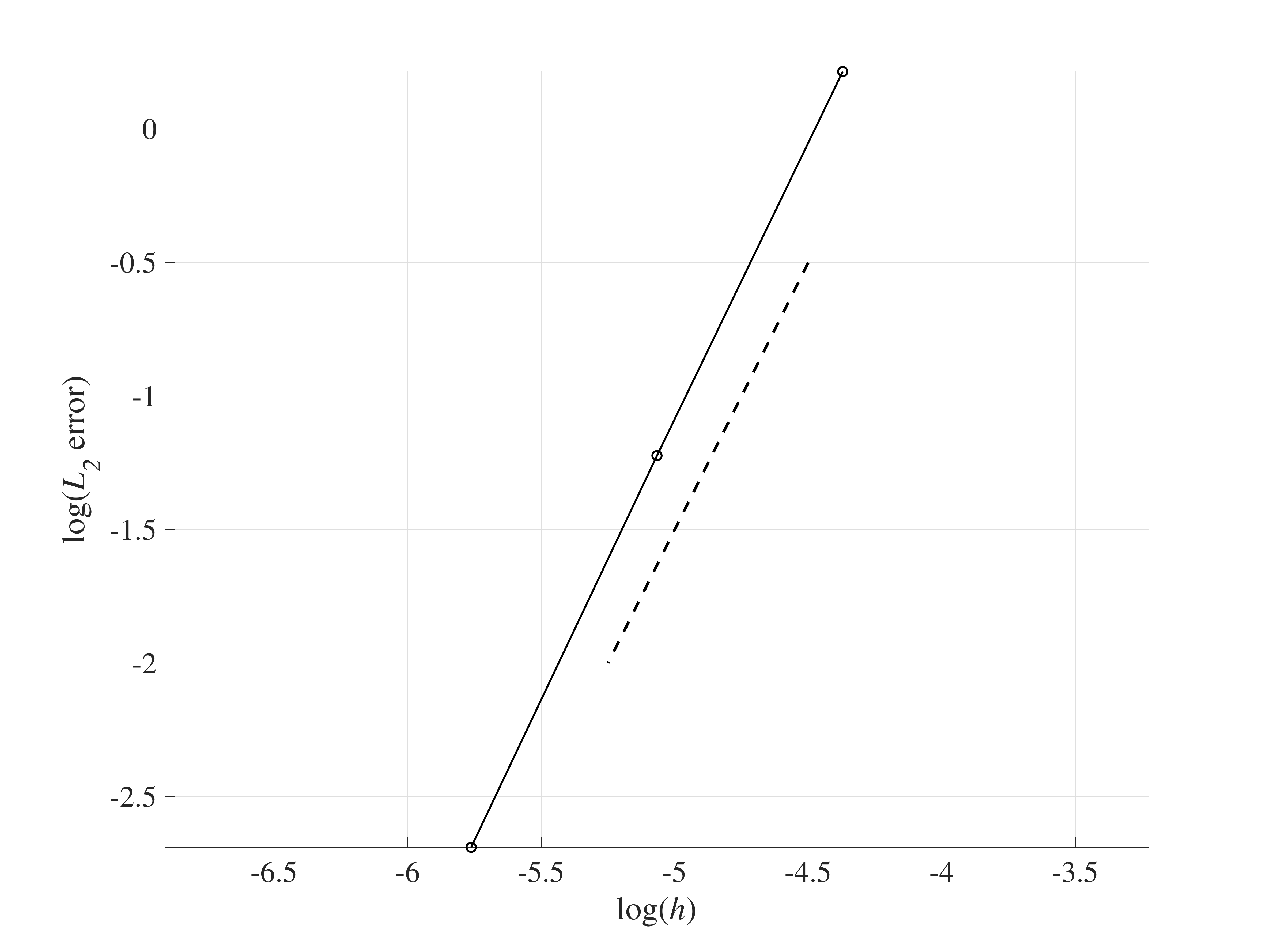}
\caption{Convergence for the geometry of Fig. \ref{fig:cube}. Dotted line has inclination 2:1.}\label{fig:cubeconv}
\end{figure}

\begin{figure}
\centering
\includegraphics[width=0.7\linewidth]{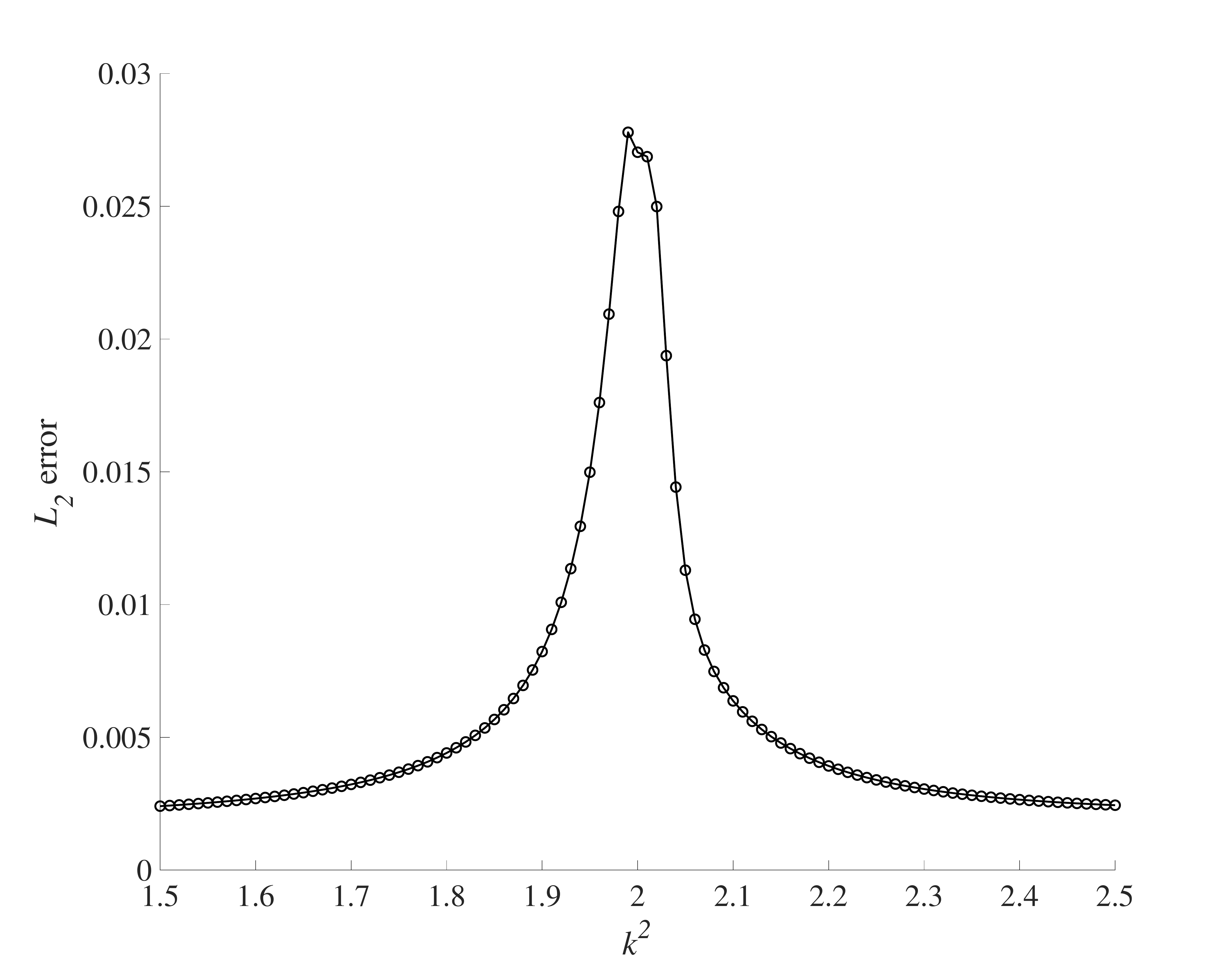}
\caption{Error close to the lowest eigenvalue of the Laplace--Beltrami operator.}\label{fig:eigen1}
\end{figure}
\begin{figure}
\centering
\includegraphics[width=0.7\linewidth]{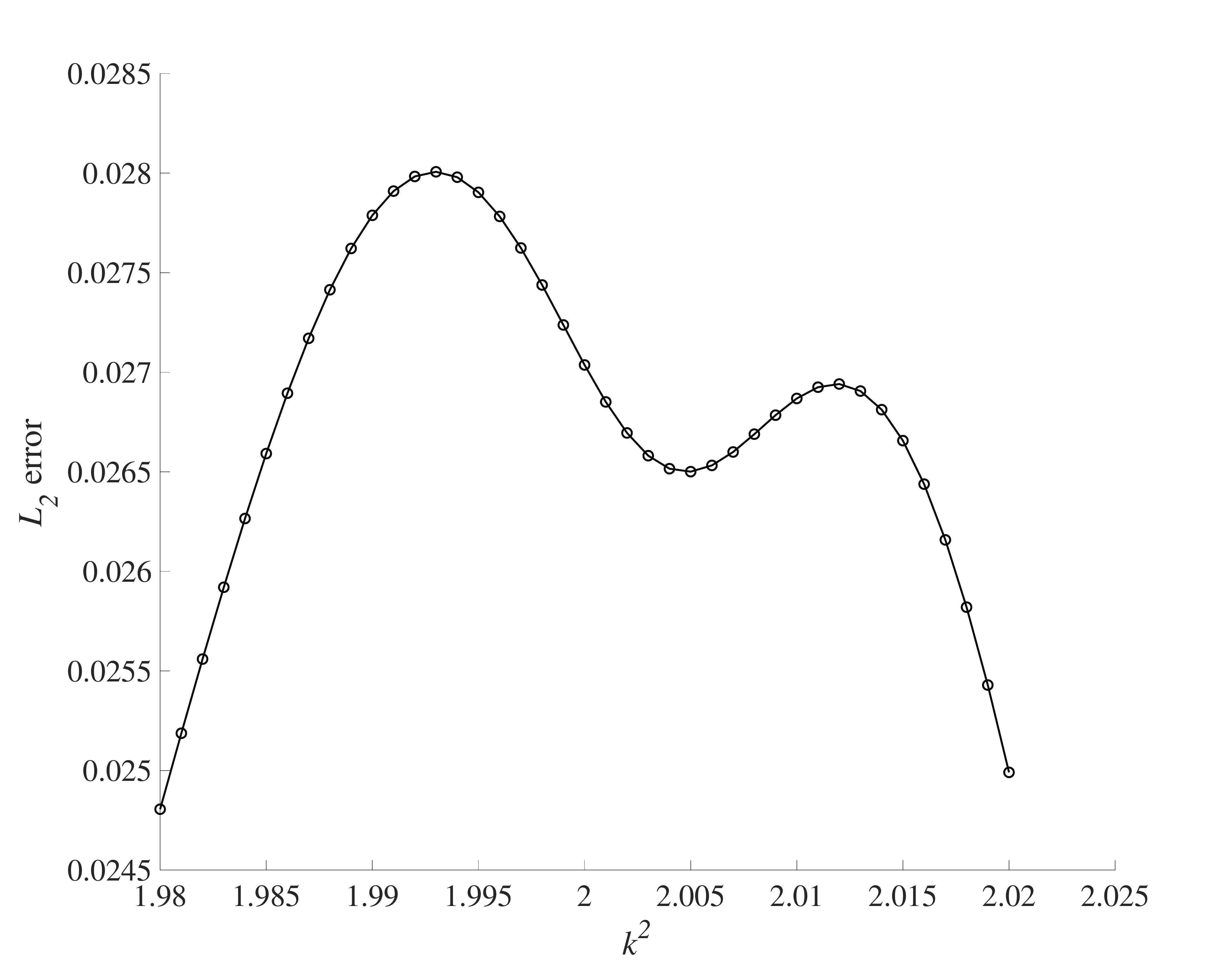}
\caption{Close-up of the error, with stabilization.}\label{fig:eigen2}
\end{figure}
\begin{figure}
\centering
\includegraphics[width=0.7\linewidth]{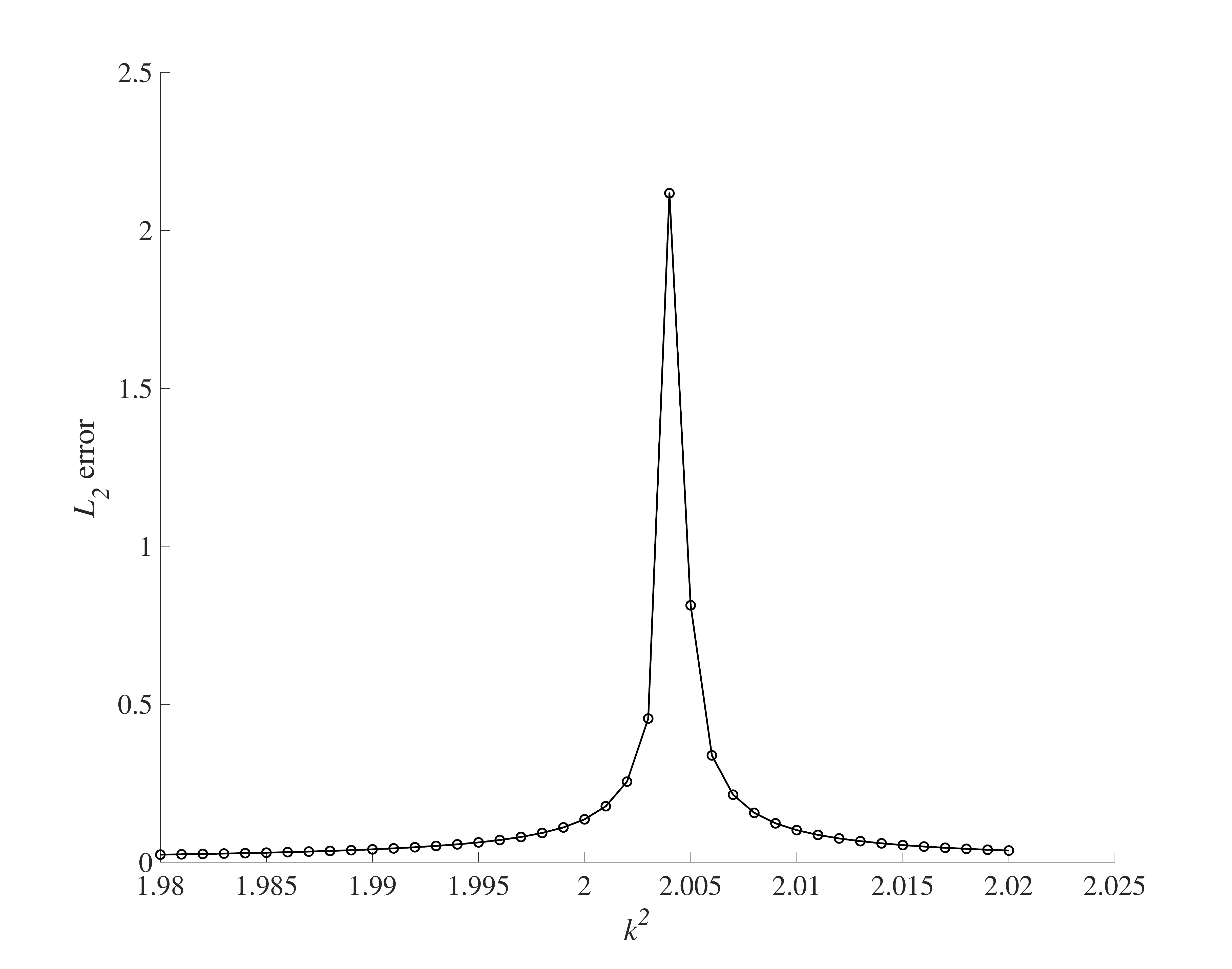}
\caption{Close-up of the error, without stabilization.}\label{fig:eigen3}
\end{figure}

\end{document}